\documentclass{amsart}
\usepackage{amssymb,amsmath,latexsym}
\usepackage{times}
\usepackage{tikz}

\numberwithin{equation}{section}

\newtheorem{thm}{Theorem}[section]

\newtheorem{cor}[thm]{Corollary}

\newtheorem{lemma}[thm]{Lemma}

\newcommand{\ssp}{\hspace{1pt}}
\newcommand{\del}{\backslash}
\newcommand{\cl}{\mathrm{cl}}
\newcommand{\meet}{\land}
\newcommand{\join}{\lor}
\newcommand{\cov}{\mbox{$\,\prec\!\!\cdot\,$}}

\newcommand{\mcJ}{\mathcal{J}}

\newcommand{\mcA}{\mathcal{A}}
\newcommand{\mcB}{\mathcal{B}}
\newcommand{\mcC}{\mathcal{C}}

\title{Lattices Related to Extensions of Presentations of Transversal
  Matroids}

\author[J.~Bonin]{Joseph E.~Bonin}
\address{Department of Mathematics\\ The George Washington University\\
  Washington, D.C. 20052, USA} \email{jbonin@gwu.edu}



\date{\today}

\begin{document}

\begin{abstract}
  For a presentation $\mcA$ of a transversal matroid $M$, we study the
  set $T_\mcA$ of single-element transversal extensions of $M$ that
  have presentations that extend $\mcA$; we order these extensions by
  the weak order.  We show that $T_\mcA$ is a distributive lattice,
  and that each finite distributive lattice is isomorphic to $T_\mcA$
  for some presentation $\mcA$ of some transversal matroid $M$.  We
  show that $T_\mcA\cap T_\mcB$, for any two presentations $\mcA$ and
  $\mcB$ of $M$, is a sublattice of both $T_\mcA$ and $T_\mcB$.  We
  prove sharp upper bounds on $|T_\mcA|$ for presentations $\mcA$ of
  rank less than $r(M)$ in the order on presentations; we also give a
  sharp upper bound on $|T_\mcA\cap T_\mcB|$.  The main tool we
  introduce to study $T_\mcA$ is the lattice $L_\mcA$ of closed sets
  of a certain closure operator on the lattice of subsets of
  $\{1,2,\ldots,r(M)\}$.
\end{abstract}

\maketitle

\markboth{Several Lattices Related to Extensions and Presentations of
  Transversal Matroids}{Several Lattices Related to Extensions and
  Presentations of Transversal Matroids}

\section{Introduction}\label{sec:intro}

We continue the investigation, begun in \cite{extpres}, of the extent
to which a presentation $\mcA$ of a transversal matroid $M$ limits the
single-element transversal extensions of $M$ that can be obtained by
extending $\mcA$.  The following analogy may help orient readers.  A
matrix $A$, over a field $\mathbb{F}$, that represents a matroid $M$
may contain extraneous information; this can limit which
$\mathbb{F}$-representable single-element extensions of $M$ can be
represented by extending (i.e., adjoining another column to) $A$.  For
instance, for the rank-$3$ uniform matroid $U_{3,6}$, partition
$E(U_{3,6})$ into three $2$-point lines, $L_1$, $L_2$, and $L_3$.  Let
$A$ be a $3\times 6$ matrix, over $\mathbb{F}$, that represents
$U_{3,6}$.  The line $L_i$ is represented by a pair of columns of $A$,
which span a $2$-dimensional subspace $V_i$ of $\mathbb{F}^3$.  While
$V_i\cap V_j$, for $\{i,j\}\subset \{1,2,3\}$, has dimension $1$
(since the corresponding lines of $U_{3,6}$ are coplanar), the
intersection $V_1\cap V_2\cap V_3$ can, in general, have dimension
either $0$ or $1$: this dimension is extraneous. If $\dim(V_1\cap
V_2\cap V_3)$ is $1$, then no extension of $A$ represents the
extension of $M$ that has an element on, say, $L_1$ and $L_2$ but not
$L_3$; otherwise, no extension of $A$ represents the extension of $M$
that has a non-loop on all three lines.  (The underlying problem is
the lack of unique representability, which is a major complicating
factor for research on representable matroids.  See Oxley
\cite[Section 14.6]{oxley}.)  In this paper, we consider such
problems, but for transversal matroids in place of
$\mathbb{F}$-representable matroids, and presentations in place of
matrix representations.

A transversal matroid can be given by a presentation, which is a
sequence of sets whose partial transversals are the independent sets.
In \cite{extpres}, we introduced the ordered set $T_\mcA$ of
transversal single-element extensions of $M$ that have presentations
that extend $\mcA$ (i.e., the new element is adjoined to some of the
sets in $\mcA$), where we order extensions by the weak order.  In
Section \ref{sec:lattice}, we introduce a new tool for studying
$T_\mcA$: given a presentation $\mcA$ of a transversal matroid $M$
with the number, $|\mcA|$, of terms in the sequence $\mcA$ being the
rank, $r$, of $M$, we define a closure operator on the lattice
$2^{[r]}$ of subsets of the set $[r]=\{1,2,\ldots,r\}$, and we show
that the resulting lattice $L_\mcA$ of closed sets is a (necessarily
distributive) sublattice of $2^{[r]}$ that is isomorphic to $T_\mcA$.
While they are isomorphic, $L_\mcA$ is often simpler to work with than
is $T_\mcA$.  We prove some basic properties of the lattice $L_\mcA$,
give several descriptions of its elements, show that every
distributive lattice is isomorphic to $L_\mcA$, and so to $T_\mcA$,
for a suitable choice of $M$ and $\mcA$, and we interpret the join-
and meet-irreducible elements of $L_\mcA$.  We show that if $\mcA$ and
$\mcB$ are both presentations of $M$, then $T_\mcA\cap T_\mcB$ is a
sublattice of $T_\mcA$ and of $T_\mcB$.  In \cite{extpres}, we showed
that $|T_\mcA| = 2^r$ if and only if the presentation $\mcA$ of $M$ is
minimal in the natural order on the presentations of $M$; using
$L_\mcA$, in Section \ref{sec:applic} we prove upper bounds on
$|T_\mcA|$ for the next $r$ lowest ranks in this order.  We also show
that $|T_\mcA\cap T_\mcB|\leq \frac{3}{4} \cdot 2^r$ whenever
presentations $\mcA$ and $\mcB$ of $M$ differ by more than just the
order of the sets.

The relevant background is recalled in the next section.  See Brualdi
\cite{brualdi} for more about transversal matroids, and Oxley
\cite{oxley} for other matroid background.

\section{Background}\label{sec:background}

A \emph{set system} $\mcA=(A_i\,:\,i\in[r])$ on a set $E$ is a
sequence of subsets of $E$.  A \emph{partial transversal} of $\mcA$ is
a subset $X$ of $E$ for which there is an injection $\phi:X\rightarrow
[r]$ with $e\in A_{\phi(e)}$ for all $e\in X$; such an injection is an
\emph{$\mcA$-matching of $X$ into $[r]$}.  Edmonds and
Fulkerson~\cite{ef} showed that the partial transversals of $\mcA$ are
the independent sets of a matroid on $E$; we say that $\mcA$ is a
\emph{presentation} of this \emph{transversal matroid} $M[\mcA]$.

The first lemma is an easy observation.

\begin{lemma}\label{lem:restr}
  Let $M$ be $M[\mcA]$ with $\mcA =(A_i\,:\,i\in [r])$.  For any
  subset $X$ of $E(M)$, the restriction $M|X$ is transversal and
  $(A_i\cap X\,:\,i\in [r])$ is a presentation of $M|X$.
\end{lemma}

We focus on presentations $(A_i\,:\,i\in[r])$ of $M$ that are of the
type guaranteed by the first part of Lemma \ref{lem:r}, that is,
$r=r(M)$; the second part of the lemma explains why other
presentations are not substantially different.

\begin{lemma}\label{lem:r}
  Each transversal matroid $M$ has a presentation $\mcA$ with
  $|\mcA|=r(M)$.  If $M$ has no coloops, then all presentations of $M$
  have exactly $r(M)$ nonempty sets (counting multiplicity).
\end{lemma}

Given a presentation $\mcA=(A_i\,:\,i\in [r])$ of a transversal
matroid $M$ and a subset $X$ of $E(M)$, the $\mcA$-\emph{support},
$s_{\mcA}(X)$, \emph{of} $X$ is $$s_{\mcA}(X)=\{i\,:\,X\cap
A_i\ne\emptyset\}.$$ A \emph{cyclic set} in a matroid $M$ is a
(possibly empty) union of circuits; thus, $X\subseteq E(M)$ is cyclic
if and only if $M|X$ has no coloops.  Lemmas \ref{lem:restr} and
\ref{lem:r} give the next result.

\begin{cor}\label{cor:supp2} If $X$ is a cyclic set of $M[\mcA]$,
  then $|s_{\mcA}(X)|=r(X)$.
\end{cor}

By Hall's theorem \cite[Theorem VIII.8.20]{aigner}, a subset $Y$ of
$E(M)$ is independent in $M$ if and only if $|s_\mcA(Z)|\geq |Z|$ for
all subsets $Z$ of $Y$.  One can prove the next lemma from this.

\begin{lemma} \label{lem:supp1}
  Let $\mcA$ be a presentation of $M$.
  \begin{enumerate}
  \item For any circuit $C$ of $M$ and element $e\in C$, we have 
    $$|s_{\mcA}(C)|=|s_{\mcA}(C-\{e\})|=r(C)=|C|-1,$$ so $s_{\mcA}(C)=
    s_{\mcA}(C-\{e\})$.
  \item If $X\subseteq E(M)$ with $|s_{\mcA}(X)|=r(X)$, then its
    closure, $\cl(X)$, is $$\cl(X) = \{e\,:\, s_{\mcA}(e)\subseteq
    s_{\mcA}(X)\}.$$
  \end{enumerate}
\end{lemma}

Extending a presentation $\mcA=(A_i:i\in [r])$ of a transversal
matroid $M$ consists of adjoining an element $x$ that is not in $E(M)$
to some of the sets in $\mcA$.  More precisely, for an element
$x\not\in E(M)$ and a subset $I$ of $[r]$, we let $\mcA^I$ be
$(A^I_i:i\in[r])$ where
$$A^I_i=\left\{
  \begin{array}{ll}
    A_i\cup \{x\}, &\mbox{ if } i\in I,\\ 
    A_i, &\mbox{ otherwise.} 
  \end{array}
\right.$$ The matroid $M[\mcA^I]$ on the set $E(M)\cup \{x\}$ is a
rank-preserving single-element extension of $M$.  (This is the only
type of extension we consider, so below we omit the adjectives
``rank-preserving'' and ``single-element''.)  Throughout this paper,
we reserve $x$ for the element by which we extend a matroid.

We will use principal extensions of matroids, which we now recall.
For any matroid $M$ (not necessarily transversal), a subset $Y$ of
$E(M)$, and an element $x$ that is not in $E(M)$, the \emph{principal
  extension} $M+_{_Y}x$ of $M$ is the matroid on $E(M)\cup \{x\}$ with the
rank function $r'$ where, for $Z\subseteq E(M)$, we have $r'
(Z)=r_M(Z)$ and
$$r'(Z \cup \{x\}) = \left\{
  \begin{array}{ll}
    r_M (Z),  & \text{if}\,\, Y \subseteq \cl_M(Z),\\
    r_M (Z) + 1, & \text{otherwise.}  
\end{array}\right.  $$
Thus, $M+_{_Y}x =M+_{_{Y'}}x$ whenever $\cl_M(Y) =\cl_M(Y')$.
Geometrically, $M+_{_Y}x$ is formed by putting $x$ freely in the flat
$\cl_M(Y)$.  A routine argument using matchings and part (2) of Lemma
\ref{lem:supp1} yields the following result.

\begin{lemma}\label{lem:princip} 
  Let $\mcA$ be a presentation of a transversal matroid $M$.  If $Y$
  is a subset of $E(M)$ with $|s_\mcA(Y)|=r(Y)$, then
  $M[\mcA^{s_\mcA(Y)}]$ is the principal extension $M+_{_Y}x$, and,
  relative to containment, the least cyclic flat of
  $M[\mcA^{s_\mcA(Y)}]$ that contains $x$ is $\cl_M(Y)\cup \{x\}$.
\end{lemma}

A transversal matroid typically has many presentations, and there is a
natural order on them.  A mild variant of the customary order on
presentations best meets our needs.  For presentations
$\mcA=(A_i:i\in[r])$ and $\mcB= (B_i:i\in[r])$ of $M$, we set
$\mcA\preceq \mcB$ if $A_i\subseteq B_i$ for all $i\in [r]$.  We write
$\mcA\prec \mcB$ if, in addition, at least one of these inclusions is
strict.  We say that $\mcB$ \emph{covers} $\mcA$, and we write
$\mcA\cov\mcB$, if $\mcA\prec \mcB$ and there is no presentation
$\mcC$ of $M$ with $\mcA \prec \mcC \prec \mcB$.  (The customary order
identifies $(A_i\,:\,i\in[r])$ and $(A_{\tau(i)}\,:\,i\in [r])$ for
any permutation $\tau$ of $[r]$, and so sets $\mcA\leq \mcB$ if, up to
re-indexing, $A_i\subseteq B_i$ for all $i\in [r]$.)

Mason \cite{JHMThesis} showed that if $(A_i\,:\,i\in[r])$ and
$(B_i\,:\,i\in[r])$ are maximal presentations of the same transversal
matroid, then there is a permutation $\tau$ of $[r]$ with
$A_{\tau(i)}=B_i$ for all $i\in [r]$.  (Minimal presentations, in
contrast, are often more varied.)  The next lemma, which is due to
Bondy and Welsh \cite{bw} and plays important roles in this paper,
gives a constructive way to find the maximal presentations of a
transversal matroid.

\begin{lemma}\label{lem:max}
  Let $\mcA =(A_i:i\in[r])$ be a presentation of $M$.  Let $i$ be in
  $[r]$ and $e$ in $E(M)-A_i$.  The following statements are
  equivalent:
  \begin{enumerate}
  \item the set system obtained from $\mcA$ by replacing $A_i$ by
    $A_i\cup \{e\}$ is also a presentation of $M$, and
  \item $e$ is a coloop of the deletion $M\del A_i$.
  \end{enumerate}
\end{lemma}

A routine argument shows that the complement $E(M)-A_i$ of any set
$A_i$ in $\mcA$ is a flat of $M[\mcA]$.  By Lemma \ref{lem:max}, the
complement of each set in a maximal presentation of $M$ is a cyclic
flat of $M$.  Bondy and Welsh \cite{bw} and Las Vergnas \cite{mlv}
proved the next result about the sets in minimal presentations.

\begin{lemma}\label{lem:min}
  A presentation $(C_i:i\in[r])$ of $M$ is minimal if and only if each
  set $C_i$ is a cocircuit of $M$, that is, $E(M)-C_i$ is a hyperplane
  of $M$.
\end{lemma}

Thus, $(C_i:i\in[r])$ is minimal if and only if $r(M\del C_i)=r-1$ for
all $i\in [r]$.  The next result, by Brualdi and Dinolt \cite{brudin},
follows from the last two lemmas.

\begin{lemma}
  If $\mcA=(A_i:i\in[r])$ is a presentation of $M$ and
  $\mcC=(C_i:i\in[r])$ is a minimal presentation of $M$ with
  $\mcC\preceq\mcA$, then $$|A_i-C_i| = r(M\del C_i)-r(M\del A_i) =
  r-1-r(M\del A_i).$$
\end{lemma}

\begin{cor}\label{cor:presgraded}
  The ordered set of presentations of a rank-$r$ transversal matroid
  $M$ is ranked; the rank of a presentation $(A_i:i\in[r])$
  is $$r(r-1)-\sum_{i=1}^r r(M\del A_i).$$
\end{cor}

This corollary applies to both the order we focus on,
$\mcA\preceq\mcB$, and the more customary order, $\mcA\leq \mcB$; the
rank of a presentation is the same in both orders.

The \emph{weak order} $\leq_w$ on matroids on the same set $E$ is
defined as follows: $M\leq_w N$ if $r_M(X)\leq r_N(X)$ for all subsets
$X$ of $E$; equivalently, every independent set of $M$ is independent
in $N$.  This captures the idea that $N$ is freer than $M$.  The next
two lemmas are simple but useful observations.

\begin{lemma}\label{lem:inclusinggivesweak}
  Let $M=M[(A_i\,:\,i\in [r])]$ and $N=M[(B_i\,:\,i\in [r])]$, where
  $M$ and $N$ are defined on the same set.  If $A_i\subseteq B_i$ for
  all $i\in[r]$, then $M\leq_w N$.
\end{lemma}

\begin{lemma}\label{lem:useweak}
  Assume that $M\leq_w N$ and $M\del e = N\del e$.  If $e$ is a coloop
  of $M$, then $e$ is a coloop of $N$, and so $M=N$.
\end{lemma}

Lastly, we recall how to think of transversal matroids geometrically
and to give affine representations of those of low rank, as in Figures
\ref{fig:example} and \ref{fig:example2}.  A set system
$\mcA=(A_i\,:\,i\in[r])$ on $E$ can be encoded by a $0$-$1$ matrix
with $r$ rows whose columns are indexed by the elements of $E$ in
which the $i,e$ entry is $1$ if and only if $e\in A_i$.  If we replace
the $1$s in this matrix by distinct variables, say over $\mathbb{R}$,
then it follows from the permutation expansion of determinants that
the linearly independent columns are precisely the partial
transversals of $\mcA$, so this is a matrix representation of
$M[\mcA]$.  One can in turn replace the variables by non-negative real
numbers and preserve which square submatrices have nonzero
determinants; one can also scale the columns so that the sum of the
entries in each nonzero column is $1$.  In this way, each non-loop of
$M$ is represented by a point in the convex hull of the standard basis
vectors.  This yields the following geometric picture: label the
vertices of a simplex $1,2,\ldots,r$ and think of associating $A_i$ to
the $i$-th vertex, then place each point $e$ of $E$ freely (relative
to the other points) in the face of the simplex spanned by
$s_\mcA(e)$.

\section{A closure operator and two isomorphic distributive
  lattices}\label{sec:lattice}

Let $\mcA$ be a presentation of $M$.  In \cite{extpres}, we introduced
the ordered set $T_\mcA$ of transversal extensions of $M$ that have
presentations that extend $\mcA$, ordering $T_\mcA$ by the weak order.
As the results in this paper demonstrate, the lattice $L_\mcA$ of
subsets of $[r(M)]$ that we define in this section and show to be
isomorphic to $T_\mcA$ is very useful for studying $T_\mcA$.

Recall that we consider only single-element rank-preserving
extensions.  Also, $x$ always denotes the element by which we extend a
matroid.

\subsection{The lattice $L_\mcA$}

The first lattice we discuss is the lattice of closed sets for a
closure operator that we introduce below, so we first recall closure
operators (see, e.g., \cite[p.~49]{aigner}).  A \emph{closure
  operator} on a set $S$ is a map $\sigma:2^S\to 2^S$ for which
\begin{enumerate}
\item $X\subseteq \sigma(X)$ for all $X\subseteq S$,
\item if $X\subseteq Y\subseteq S$, then $\sigma(X)\subseteq
  \sigma(Y)$, and
\item $\sigma(\sigma(X))= \sigma(X)$ for all $X\subseteq S$.
\end{enumerate}
Given a closure operator $\sigma:2^S\to 2^S$, a \emph{$\sigma$-closed
  set} is a subset $X$ of $S$ with $\sigma(X)=X$.  The set of
$\sigma$-closed sets, ordered by containment, is a lattice; join and
meet are given by $X\join Y = \sigma(X\cup Y)$ and $X\meet Y = X\cap
Y$. By property (1), the set $S$ is $\sigma$-closed.

Let $\mcA$ be a presentation of a rank-$r$ transversal matroid $M$.
By Lemma \ref{lem:max}, for each subset $I$ of $[r]$, there is a
greatest subset $K$ of $[r]$, relative to containment, for which
$M[\mcA^I]=M[\mcA^K]$, namely $$K=I\cup \{k\in[r]-I\,:\,x \text{ is a
  coloop of } (M[\mcA^I])\del A_k\};$$ define a map
$\sigma_\mcA:2^{[r]}\to 2^{[r]}$ by setting $\sigma_\mcA(I)=K$.  We
next show that $\sigma_\mcA$ is a closure operator.  We use $L_\mcA$
to denote the lattice of $\sigma_\mcA$-closed sets.  See Figure
\ref{fig:example} for examples.

\begin{figure}
\centering
\begin{tikzpicture}[scale=1]
  \node[inner sep = 0mm] (1) at (0,1) {};
  \node at (-1,0.7) {\small $A_1=\{a,b,c\}$};
  \node[inner sep = 0mm] (1') at (1.45,1) {};
  \node[inner sep = 0mm] (2) at (1.5,0.2) {};
  \node at (1.5,-0.1) {\small $A_3=\{d,e,f,g,h,i\}$};
  \node[inner sep = 0mm] (3) at (3,1) {};
  \node at (3.7,1.35) {\small $A_4=\{g,h,i\}$};
  \node[inner sep = 0mm] (3') at (1.55,1) {};
  \node[inner sep = 0mm] (4) at (1.5,2.25) {};
  \node at (1.5,2.5) {\small $A'_2=\{b,c,d,e,f\}$};
\foreach \from/\to in {1/2,1/1',3'/3,1/4,2/3,2/4,3/4}
\draw(\from)--(\to);

\filldraw (0,1) node[above left] {$a$} circle  (2pt);
\filldraw (0.75,1.625) node[above left] {$b$} circle  (2pt);
\filldraw (1.1,1.916) node[above left] {$c$} circle  (2pt);
\filldraw (1.5,1.7) node[right] {$d$} circle  (2pt);
\filldraw (1.5,1.25) node[right] {$e$} circle  (2pt);
\filldraw (1.5,0.75) node[right] {$f$} circle  (2pt);
\filldraw (1.9,0.413) node[below right] {$g$} circle  (2pt);
\filldraw (2.25,0.6) node[below right] {$h$} circle  (2pt);
\filldraw (2.6,0.786) node[below right] {$i$} circle  (2pt);

  \node[inner sep = 0.3mm] (em) at (6.8,0) {\small $\emptyset$};
  \node[inner sep = 0.3mm]  (s1) at (5.6,0.75) {\small $\{1\}$};
  \node[inner sep = 0.3mm] (s2) at (6.8,0.75) {\small $\{2\}$};
  \node[inner sep = 0.3mm] (s3) at (8,0.75) {\small $\{3\}$};
  \node[inner sep = 0.3mm] (s12) at (5.6,1.5) {\small $\{1,2\}$};
  \node[inner sep = 0.3mm] (s13) at (6.8,1.5) {\small $\{1,3\}$};
  \node[inner sep = 0.3mm] (s23) at (8,1.5) {\small $\{2,3\}$};
  \node[inner sep = 0.3mm] (s34) at (9.2,1.5) {\small $\{3,4\}$};
  \node[inner sep = 0.3mm] (s123) at (6.8,2.25) {\small $\{1,2,3\}$};
  \node[inner sep = 0.3mm] (s134) at (8,2.25) {\small $\{1,3,4\}$};
  \node[inner sep = 0.3mm] (s234) at (9.2,2.25) {\small $\{2,3,4\}$};
  \node[inner sep = 0.3mm] (s1234) at (8,3) {\small $\{1,2,3,4\}$};
\foreach \from/\to in
{em/s1,em/s2,em/s3,s1/s12,s1/s13,s2/s23,s2/s12,s3/s23,s3/s13,s3/s34,s12/s123,
s13/s123,s23/s123,s13/s134,s34/s134,s23/s234,s34/s234,s123/s1234,
s134/s1234,s234/s1234}
\draw(\from)--(\to);

  \node[inner sep = 0mm] (u1) at (0,5) {};
  \node at (-1,4.7) {\small $A_1=\{a,b,c\}$};
  \node[inner sep = 0mm] (u1') at (1.45,5) {};
  \node[inner sep = 0mm] (u2) at (1.5,4.2) {};
  \node at (1.5,3.9) {\small $A_3=\{d,e,f,g,h,i\}$};
  \node[inner sep = 0mm] (u3) at (3,5) {};
  \node at (3.7,5.35) {\small $A_4=\{g,h,i\}$};
  \node[inner sep = 0mm] (u3') at (1.55,5) {};
  \node[inner sep = 0mm] (u4) at (1.5,6.25) {};
  \node at (1.5,6.5) {\small $A_2=\{a,b,c,d,e,f\}$};
\foreach \from/\to in {u1/u2,u1/u1',u3'/u3,u1/u4,u2/u3,u2/u4,u3/u4}
\draw(\from)--(\to);

\filldraw (0.4,5.33) node[above left] {$a$} circle  (2pt);
\filldraw (0.75,5.625) node[above left] {$b$} circle  (2pt);
\filldraw (1.1,5.916) node[above left] {$c$} circle  (2pt);
\filldraw (1.5,5.7) node[right] {$d$} circle  (2pt);
\filldraw (1.5,5.25) node[right] {$e$} circle  (2pt);
\filldraw (1.5,4.75) node[right] {$f$} circle  (2pt);
\filldraw (1.9,4.413) node[below right] {$g$} circle  (2pt);
\filldraw (2.25,4.6) node[below right] {$h$} circle  (2pt);
\filldraw (2.6,4.786) node[below right] {$i$} circle  (2pt);

  \node[inner sep = 0.3mm] (em) at (7.5,4) {\small $\emptyset$};
  \node[inner sep = 0.3mm] (t2) at (6.75,4.75) {\small $\{2\}$};
  \node[inner sep = 0.3mm] (t3) at (8.25,4.75) {\small $\{3\}$};
  \node[inner sep = 0.3mm] (t12) at (6,5.5) {\small $\{1,2\}$};
  \node[inner sep = 0.3mm] (t23) at (7.5,5.5) {\small $\{2,3\}$};
  \node[inner sep = 0.3mm] (t34) at (9,5.5) {\small $\{3,4\}$};
  \node[inner sep = 0.3mm] (t123) at (6.75,6.25) {\small $\{1,2,3\}$};
  \node[inner sep = 0.3mm] (t234) at (8.25,6.25) {\small $\{2,3,4\}$};
  \node[inner sep = 0.3mm] (t1234) at (7.5,7) {\small $\{1,2,3,4\}$};
\foreach \from/\to in
{em/t2,em/t3,t2/t23,t2/t12,t3/t23,t3/t34,t12/t123,
t23/t123,t23/t234,t34/t234,t123/t1234,t234/t1234}
\draw(\from)--(\to);
\end{tikzpicture}
\caption{Two presentations $\mcA$ of a transversal matroid $M$, along
  with the associated lattices $L_\mcA$.}
\label{fig:example}
\end{figure}
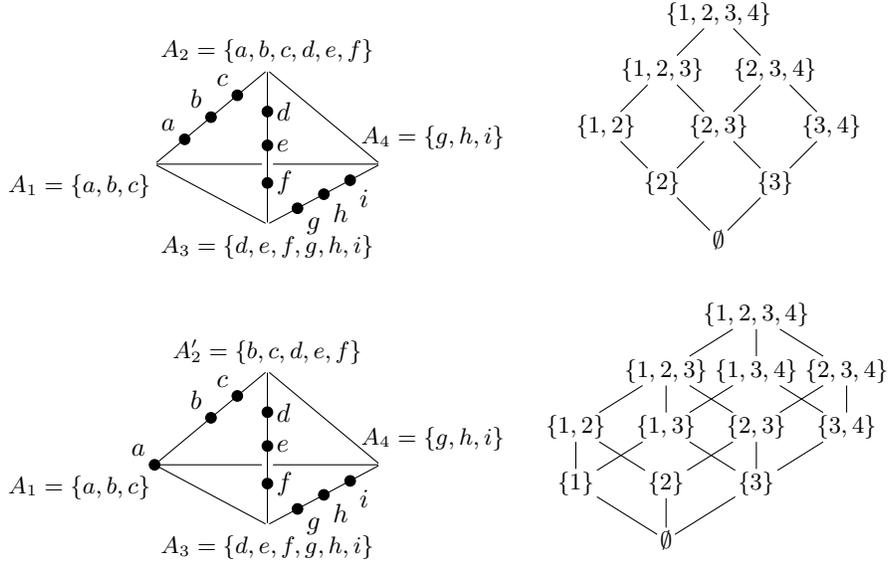

\begin{thm}\label{thm:distriblattice}
  For any presentation $\mcA=(A_i\,:\,i\in[r])$ of a transversal
  matroid $M$, the map $\sigma_\mcA$ defined above is a closure
  operator on $[r]$.  The join in the lattice $L_\mcA$ of
  $\sigma_\mcA$-closed sets is given by $I\join J = I\cup J$, so
  $L_\mcA$ is distributive.  Both $\emptyset$ and $[r]$ are in
  $L_\mcA$.
\end{thm}

\begin{proof}
  Properties (1) and (3) of closure operators clearly hold.  For
  property (2), assume $I\subseteq J\subseteq [r]$ and $h\in
  \sigma_\mcA(I)-I$, so $x$ is a coloop of $M[\mcA^I]\del A_h$.  Lemma
  \ref{lem:inclusinggivesweak} gives $M[\mcA^I]\del A_h\leq_w
  M[\mcA^J]\del A_h$, so $x$ is a coloop of $M[\mcA^J]\del A_h$ by
  Lemma \ref{lem:useweak}, so $h\in \sigma(J)$, as needed.

  Let $I$ and $J$ be in $L_\mcA$.  Their meet, $I\meet J$, is $I\cap
  J$ since, as noted above, this holds for any closure operator.  We
  claim that $I\join J = I\cup J$.  (The fact that $L_\mcA$ is
  distributive then follows since union and intersection distribute
  over each other.)  Since $I$ and $J$ are in $L_\mcA$,
  \begin{enumerate}
  \item if $h\in [r]-I$, then $x$ is not a coloop of $M[\mcA^I]\del
    A_h$, and
  \item if $h\in [r]-J$, then $x$ is not a coloop of $M[\mcA^J]\del
    A_h$.
  \end{enumerate}
  Note that the following two statements are equivalent: (i) $I\join J
  = I\cup J$ and (ii) $I\cup J$ is $\sigma_\mcA$-closed.  To prove
  statement (ii), let $h$ be in $[r]-(I\cup J)$ and let $Z$ be a basis
  of $M\del A_h$.  If $x$ were a coloop of $M[\mcA^{I\cup J}]\del
  A_h$, then there would be an $\mcA^{I\cup J}$-matching $\phi:Z\cup
  \{x\}\to [r]$.  Either $\phi(x)\in I$ or $\phi(x)\in J$; if
  $\phi(x)\in I$, then $\phi$ shows that $Z\cup \{x\}$ is independent
  in $M[\mcA^I ]\del A_h$, contrary to item (1) above; similarly,
  $\phi(x)\in J$ contradicts item (2).  Thus, as needed, $x$ is not a
  coloop of $M[\mcA^{I\cup J}]\del A_h$.

  Note that $\emptyset$ is in $L_\mcA$ since $x$ is a loop of
  $M[\mcA^I]$ if and only if $I =\emptyset$.
\end{proof}

We now show how the order on presentations relates to the lattices of
closed sets.

\begin{thm}\label{thm:altorder}
  For two presentations $\mcA=(A_i\,:\,i\in[r])$ and $\mcB
  =(B_i\,:\,i\in[r])$ of $M$, if $\mcA\preceq \mcB$, then $L_\mcB$ is
  a sublattice of $L_\mcA$ and $M[\mcA^I]=M[\mcB^I]$ for all $I\in
  L_\mcB$.
\end{thm}

\begin{proof}
  Fix $I$ in $L_\mcB$.  Set $M_\mcB =M[\mcB^I]$ and
  $M_\mcA=M[\mcA^I]$.  For $i\in [r]-I$, the element $x$ is not a
  coloop of $M_\mcB\del B_i$ since $I \in L_\mcB$.  Now $M_\mcA\del
  B_i\leq_w M_\mcB\del B_i$, so $x$ is not a coloop of $M_\mcA\del
  B_i$ by Lemma \ref{lem:useweak}, so $x$ is not a coloop of
  $M_\mcA\del A_i$.  Thus, $I\in L_\mcA$, so $L_\mcB$ is a sublattice
  of $L_\mcA.$ Lemma \ref{lem:max} and the following two claims give
  $M_\mcA = M_\mcB$:
  \begin{enumerate}
  \item for each $i\in I$, each element of $(B_i\cup \{x\})-(A_i\cup
    \{x\})$ (that is, $B_i -A_i$) is a coloop of $M_\mcA\del (A_i\cup
    \{x\})$ (that is, $M\del A_i$), and
  \item for each $i\in [r]-I$, each element of $B_i -A_i$ is a coloop
    of $M_\mcA\del A_i$.
  \end{enumerate}
  By the hypothesis and Lemma \ref{lem:max}, for all $i\in [r]$, each
  element of $B_i-A_i$ is a coloop of $M\del A_i$, so claim (1) holds.
  For claim (2), fix $i \in [r]-I$ and $y \in B_i-A_i$.  As shown
  above, $x$ is not a coloop of $M_\mcA\del B_i$; let $C$ be a circuit
  of $M_\mcA\del B_i$ with $x\in C$.  Thus, $y\not\in C$.  Assume,
  contrary to claim (2), that some circuit $C'$ of $M_\mcA\del A_i$
  contains $y$.  Now $x\in C'$ since $y$ is coloop of $M\del A_i$.  By
  strong circuit elimination, applied in $M_\mcA\del A_i$, some
  circuit $C''\subseteq (C\cup C')-\{x\}$ contains $y$; however $C''$
  is a circuit of $M\del A_i$, which contradicts $y$ being a coloop of
  $M\del A_i$.  Thus, claim (2) holds.
\end{proof}

The corollary below is a theorem from \cite{extpres}.

\begin{cor}\label{cor:recoverminresult}
  For each transversal extension $M'$ of $M$, there is a minimal
  presentation of $M$ that can be extended to a presentation of $M'$.
\end{cor}

\subsection{The lattice $T_\mcA$}\label{sec:T}

The lattice $T_\mcA$ consists of the set $\{M[\mcA^I]\,:\, I\in
L_{\mcA}\}$ of transversal extensions of $M$ that have presentations
that extend $\mcA$, which we order by the weak order.  The next result
relates  $T_\mcA$ and  $L_\mcA$.

\begin{thm}\label{thm:faithful2}
  Let $\mcA$ be a presentation of $M$.  For any sets $I$ and $J$ in
  $L_\mcA$, we have $M[\mcA^I]\leq_w M[\mcA^J]$ if and only if
  $I\subseteq J$.  Thus, the bijection $I\mapsto M[\mcA^I]$ from
  $L_\mcA$ onto $T_\mcA$ is a lattice isomorphism, so $T_\mcA$ is a
  distributive lattice.
\end{thm}

\begin{proof}
  Assume that $M[\mcA^I]\leq_w M[\mcA^J]$.  Any $\mcA^{I\cup
    J}$-matching $\phi$ of an independent set $X$ of $M[\mcA^{I\cup
    J}]$ with $x\in X$ has $\phi(x)$ in either $I$ or $J$, so $X$ is
  independent in one of $M[\mcA^I]$ and $M[\mcA^J]$, and so, by the
  assumption, in $M[\mcA^J]$.  Thus, $M[\mcA^{I\cup J}] \leq_w
  M[\mcA^J]$.  The equality $M[\mcA^J]= M[\mcA^{I\cup J}]$ now follows
  by Lemma \ref{lem:inclusinggivesweak}; thus, $J = I\cup J$ since $J$
  and $I\cup J$ are $\sigma_\mcA$-closed, so $I\subseteq J$.  The
  other implication follows from Lemma \ref{lem:inclusinggivesweak}.
\end{proof}

\begin{cor}
  For presentations $\mcA$ and $\mcB$ of $M$, if $\mcA\preceq \mcB$,
  then $T_\mcB$ is a sublattice of $T_\mcA$.
\end{cor}

The converse of the corollary fails even under the more common order
on presentations as we now show.

\medskip

\noindent
\textsc{Example 1.}  Consider the uniform matroid $U_{3,4}$ on
$\{a,b,c,d\}$ and its presentations
$$\mcA = (\{a,b,d\},  \{a,c,d\},\{b,c,d\})\quad \text{ and } \quad
\mcB = (\{a,b,c\},\{a,b,d\}, \{a,c,d\}).$$ It is easy to check that
both $T_\mcA$ and $T_\mcB$ consist of just the extension by a loop,
$U_{3,4} \oplus U_{0,0}$, and the free extension, $U_{3,5}$.  Thus,
$T_\mcA=T_\mcB =T_\mcC$, where $\mcC$ is a maximal presentation of
$U_{3,4}$, that is, $\mcC=(\{a,b,c,d\},\{a,b,c,d\},\{a,b,c,d\})$.

\medskip

From the next result, which is a reformulation of \cite[Theorem
3.1]{extpres}, we see that we cannot recover the presentation $\mcA$
from $L_\mcA$.

\begin{thm}\label{thm:charmin}
  A presentation $\mcA=(A_i:i\in[r])$ of a transversal matroid $M$ is
  minimal if and only if $L_\mcA = 2^{[r]}$, that is, $|T_\mcA|=2^r$.
\end{thm}

\begin{proof}
  If $\mcA$ is not minimal, then $r(M\del A_i)<r-1$ for some $i\in
  [r]$; thus, $x$ is a coloop of $M[\mcA^{[r]-\{i\}}]\del A_i$, so
  $[r]-\{i\}\not\in L_\mcA$.  If $\mcA$ is minimal, then $x$ is not a
  coloop of $M[\mcA^{\{i\}}]\del A_j$ for distinct $i,j\in [r]$ since
  $r(M\del A_j)=r-1$; thus, $\{i\}\in L_\mcA$, so closure under unions
  gives $L_\mcA = 2^{[r]}$.
\end{proof}

As Example 1 shows, we cannot always reconstruct the sets in $\mcA$
from $T_\mcA$; however, in some cases we can.  For the matroid in
Figure \ref{fig:example}, one can check that the sets in each of its
presentations $\mcA$ can be reconstructed from $T_\mcA$.  Also, as we
now show, for any transversal matroid $M$, the sets in each minimal
presentation $\mcA$ of $M$ can be reconstructed from $T_\mcA$.  By
Theorem \ref{thm:charmin}, from $T_\mcA$, we know whether $\mcA$ is
minimal.  If $\mcA$ is minimal, remove the free extension,
$M[\mcA^{[r]}]$, from $T_\mcA$; under the weak order, the maximal
extensions left are $M[\mcA^I]$ with $I=[r]-\{i\}$ for $i\in [r]$;
such an extension $M[\mcA^I]$ is, by Lemma \ref{lem:princip}, the
principal extension $M+_{_{H_i}}x$ of $M$, where $H_i$ is the
hyperplane of $M$ that is the complement, $E(M)-A_i$, of the cocircuit
$A_i$; also, $H_i\cup \{x\}$ is the unique cyclic hyperplane that
contains $x$; thus, we can reconstruct each set $A_i$ in $\mcA$.

\subsection{The sets in $L_\mcA$}
The results in this section, other than Corollary
\ref{cor:cyclicgivesclosed}, are used heavily in Section
\ref{sec:applic}.  We start with several characterizations of the sets
in $L_\mcA$. 

\begin{thm}\label{thm:chraracerizeclosedsets}
  For a presentation $\mcA$ of a transversal matroid $M$, the sets in
  $L_\mcA$ are
  \begin{enumerate}
  \item the sets $s_\mcA(X)$, where $X$ is an independent set of $M$
    and $|X|=|s_\mcA(X)|$, and
  \item all intersections of such sets.
  \end{enumerate}
  In particular, for $I\ \in L_\mcA$, if $\mcC_x$ is the set of all
  circuits of $M[\mcA^I]$ that contain $x$, then
  \begin{equation}\label{eq:cirsupint}
    I=\bigcap_{C\in\mcC_x}s_\mcA(C-\{x\}).
  \end{equation}
  Item \emph{(1)} could be replaced by: \emph{(1$'$)} \ the sets
  $s_\mcA(Y)$ where $r(Y)=|s_\mcA(Y)|$.
\end{thm}

\begin{proof}
  Set $r=r(M)$.  First assume that $X$ satisfies condition (1).  Set
  $I=s_\mcA(X)$.  Thus, $X\cup \{x\}$ is dependent in $M[\mcA^I]$ but
  independent in $M[\mcA^{I\cup \{h\}}]$ for any $h\in [r]-I$, so $I$ is
  in $L_\mcA$.  Since $L_\mcA$ is closed under intersection, all sets
  identified above are in $L_\mcA$.

  Fix $I$ in $L_\mcA$ and let $\mcC_x$ be as defined above.  Let $X$
  be $C-\{x\}$ for some $C\in \mcC_x$, so $X$ is independent in $M$.
  Now $s_\mcA(X)=s_{\mcA^I}(X)$, and Lemma \ref{lem:supp1} gives
  $|s_{\mcA^I}(X)|=|X|$, so $|X| =|s_\mcA(X)|$.  Also, $I
  =s_{\mcA^I}(x)\subseteq s_{\mcA^I}(C) =s_{\mcA}(X)$, so to prove
  equation (\ref{eq:cirsupint}) and show that all sets in $L_\mcA$ are
  given by items (1) and (2), it suffices to show that for each $h$ in
  $[r]-I$, there is some $C_h\in \mcC_x$ with $h\not\in
  s_{\mcA}(C_h-\{x\})$.  Now $M[\mcA^I]\lneq_w M[\mcA^{I\cup \{h\}}]$,
  so some circuit, say $C_h$, of $M[\mcA^I]$ is independent in
  $M[\mcA^{I\cup \{h\}}]$.  Thus, $C_h\in \mcC_x$
  and $$|s_{\mcA^{I\cup \{h\}}}(C_h)|\geq |C_h|>| s_{\mcA^I}(C_h) |,$$
  so $h\not\in s_{\mcA^I}(C_h)$, so $h\not\in s_\mcA(C_h-\{x\})$, as
  needed.

  Item (1$'$) can replace item (1) since, by Lemma \ref{lem:supp1},
  $r(Y)=|s_\mcA(Y)|$ for a set $Y$ if and only if $|X|=|s_\mcA(X)|$
  for some (equivalently, every) basis $X$ of $M|Y$.
\end{proof}

By Lemma \ref{lem:princip}, in terms of $T_\mcA$, the extension that
corresponds to a set $s_\mcA(X)$ in item (1) of Theorem
\ref{thm:chraracerizeclosedsets} is the principal extension, $M+_X e$.

\begin{cor}\label{cor:cyclicgivesclosed}
  Let $\mcA=(A_i\,:\,i\in [r])$ be a presentation of $M$.  If
  $F_1,F_2,\ldots,F_k$ are cyclic flats of $M$, then
  $\bigcap_{i=1}^ks_\mcA(F_i)\in L_\mcA$.  If $\mcA$ is a maximal
  presentation of $M$, then $L_\mcA$ consists of all such sets (which
  include $\emptyset$), along with $[r]$.
\end{cor}

\begin{proof}
  The first assertion follows from Theorem
  \ref{thm:chraracerizeclosedsets} since cyclic flats satisfy
  condition (1$'$).  Now let $\mcA$ be maximal.  By Theorem
  \ref{thm:chraracerizeclosedsets}, it suffices to show that if $X$ is
  an independent set of $M$ with $|X|=|s_\mcA(X)|$, then $s_\mcA(X)$
  is the intersection of the $\mcA$-supports of some set of cyclic
  flats.  Since $\mcA$ is maximal, each flat $E(M)-A_h$ of $M$, with
  $h\in[r]$, is cyclic by Lemma \ref{lem:max}.  If $h\in
  [r]-s_\mcA(X)$, then $X\subseteq E(M)-A_h$, so $s_\mcA(X)\subseteq
  s_\mcA\bigl(E(M)-A_h\bigr) $; also $h\not\in
  s_\mcA\bigl(E(M)-A_h\bigr)$.  Thus, as needed,
  \begin{equation*}
    s_\mcA(X) =\bigcap_{h\in[r]-s_\mcA(X)}s_\mcA\bigl(E(M)-A_h\bigr).\qedhere
  \end{equation*}   
\end{proof}

The next result identifies some closed sets in terms of known closed
sets and supports.

\begin{cor}\label{cor:closedclosetoclosedgen}
  Let $\mcA$ be a presentation of $M$.  Fix $F\subseteq E(M)$ and
  $J\in L_\mcA$, and set $H = s_\mcA(F)-J$.  If $|H| \leq |F|$ and
  $H\subseteq s_\mcA(e)$ for all $e\in F$, then $J\cup s_\mcA(F)\in
  L_\mcA$.  In particular, if $s_\mcA(e)-\{h\}\in L_\mcA$ for some
  $e\in E(M)$ and $h\in s_\mcA(e)$, then $s_\mcA(e)\in L_\mcA$.
\end{cor}

\begin{proof}
  Since $J\in L_\mcA$, there is a set $\mcJ$ of subsets $X$ of $E(M)$,
  all satisfying condition (1) of Theorem
  \ref{thm:chraracerizeclosedsets}, with $J=\bigcap_{X\in
    \mcJ}s_\mcA(X)$.  For each set $X\in\mcJ$, form a new set $X'$ by
  adjoining any $|s_\mcA(F)-s_\mcA(X)|$ elements of $F$ to $X$.  Note
  that $X'$ is independent: match elements in $X'-X$ to
  $s_\mcA(F)-s_\mcA(X)$.  Now $s_\mcA(X')=s_\mcA(X\cup F)$ and
  $$J\cup s_\mcA(F) = \bigcap_{X'\,:X \in
    \mcJ}s_\mcA(X').$$ Also, $|X'|=|s_\mcA(X')|$.  Thus, Theorem
  \ref{thm:chraracerizeclosedsets} gives $J\cup s_\mcA(F)\in L_\mcA$.
  
  For the last assertion, take $J = s_\mcA(e)-\{h\}$ and $F=\{e\}$.
\end{proof}

The next result gives conditions under which the support of a set is,
or is not, closed.

\begin{thm}\label{thm:supportsandclosedsets}
  Let $\mcA=(A_i\,:\,i\in[r])$ and $\mcB=(B_i\,:\,i\in[r])$ be
  presentations of $M$.
  \begin{enumerate}
  \item If the presentation $\mcA$ is maximal, then $s_\mcA(X)\in
    L_\mcA$ for all $X\subseteq E(M)$.
  \item Assume $\mcA \prec \mcB$.  For $X\subseteq E(M)$, if
    $s_\mcA(X) \ne s_\mcB(X)$, then $s_\mcA(X)\not\in L_\mcB$.
  \end{enumerate}
\end{thm}

\begin{proof}
  We start with an observation.  For an element $e\in E(M)$, set $I =
  s_\mcA(e)$. Since $e$ and $x$ are in the same sets in $\mcA^I$, the
  transposition $\phi$ on $E(M)\cup \{x\}$ that switches $e$ and $x$
  is an automorphism of $M[\mcA^I]$.  Thus, $\phi$ restricted to
  $E(M)$ is an isomorphism of $M$ onto $M[\mcA^I]\del e$.

  For part (1), since $L_\mcA$ is closed under unions, it suffices to
  treat a singleton set $\{e\}$.  Since $[r]\in L_\mcA$, we may assume
  that $s_\mcA(e) \ne [r]$.  Set $I = s_\mcA(e)$ and fix $h\in [r]-I$.
  By Lemma \ref{lem:max}, since $\mcA$ is maximal, $e$ is not a coloop
  of $M\del A_h$, so, by the isomorphism above, $x$ is not a coloop of
  $M[\mcA^I]\del(A_h\cup \{e\})$.  Thus, $x$ is not a coloop of
  $M[\mcA^I]\del A_h$, so $I\in L_\mcA$.

  For part (2), set $J = s_\mcA(X)$, fix $h \in s_\mcB(X)-J$, and pick
  $e\in X$ with $h \in s_\mcB(e)$.  Set $I = s_\mcA(e)$.  Since $\mcA
  \prec \mcB$, the element $e$ is a coloop of $M\del A_h$ by Lemma
  \ref{lem:max}.  By the isomorphism above, $x$ is a coloop of
  $M[\mcA^I]\del (A_h\cup \{e\})$, and thus of $M[\mcB^J]\del
  (A_h\cup\{e\})$ by Lemma \ref{lem:useweak}, and thus of
  $M[\mcB^J]\del B_h$.  Thus, $J\not\in L_\mcB$.
\end{proof}

Let $\mcA=(A_i\,:\,i\in[r])$ be a maximal presentation of $M$.  Thus,
$s_\mcA(e)\in L_\mcA$ for all $e\in E(M)$ by Theorem
\ref{thm:supportsandclosedsets}.  The unions of the sets $s_\mcA(e)$
include the supports of all cyclic flats, but intersections of
supports of cyclic flats, which are in $L_\mcA$, need not be
intersections of the sets $s_\mcA(e)$, as the example in Figure
\ref{fig:example2} shows.  Each presentation $\mcA$ of $M$ is both
maximal and minimal, so $L_\mcA = 2^{[4]}$.  However, $\{2,3\}$ is not
an intersection of the $\mcA$-supports of singletons.  Thus, the sets
$s_\mcA(e)$ generate $L_\mcA$, but both their unions and the
intersections of such unions are needed to obtain all of $L_\mcA$.

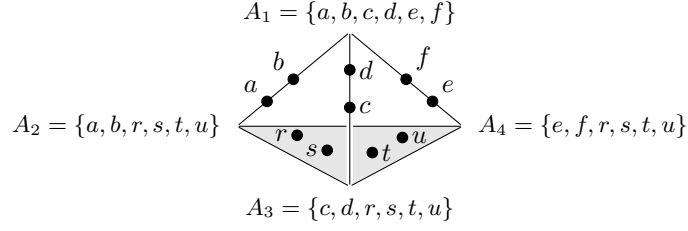
\begin{figure}
\centering
\begin{tikzpicture}[scale=1]
\draw [black!10, fill=black!10] (0.01,0.99) -- (1.45,0.99) --
(1.45,0.25) -- (0.01,0.99);
\draw [black!10, fill=black!10] (2.99,0.99) -- (1.55,0.99) --
(1.55,0.25) -- (2.99,0.99);

  \node[inner sep = 0mm] (1) at (0,1) {};
  \node at (-1.6,1) {\small $A_2=\{a,b,r,s,t,u\}$};
  \node[inner sep = 0mm] (1') at (1.48,1) {};
  \node[inner sep = 0mm] (2) at (1.5,0.2) {};
  \node at (1.5,-0.1) {\small $A_3=\{c,d,r,s,t,u\}$};
  \node[inner sep = 0mm] (3) at (3,1) {};
  \node at (4.6,1) {\small $A_4=\{e,f,r,s,t,u\}$};
  \node[inner sep = 0mm] (3') at (1.52,1) {};
  \node[inner sep = 0mm] (4) at (1.5,2.25) {};
  \node at (1.5,2.5) {\small $A_1=\{a,b,c,d,e,f\}$};
\foreach \from/\to in {1/2,1/1',3'/3,1/4,2/3,2/4,3/4}
\draw(\from)--(\to);

\filldraw (0.4,1.33) node[above left] {$a$} circle  (2pt);
\filldraw (0.75,1.625) node[above left] {$b$} circle  (2pt);
\filldraw (1.5,1.75) node[right] {$d$} circle  (2pt);
\filldraw (1.5,1.25) node[right] {$c$} circle  (2pt);
\filldraw (2.6,1.33) node[above right] {$e$} circle  (2pt);
\filldraw (2.25,1.625) node[above right] {$f$} circle  (2pt);
\filldraw (0.8,0.88) node[left] {$r$} circle  (2pt);
\filldraw (1.2,0.68) node[left] {$s$} circle  (2pt);
\filldraw (1.8,0.65) node[right] {$t$} circle  (2pt);
\filldraw (2.2,0.85) node[right] {$u$} circle  (2pt);
\end{tikzpicture}
\caption{A transversal matroid whose minimal presentations are also
  maximal.  The points $r,s,t,u$ are freely in the shaded plane.}
\label{fig:example2}
\end{figure}

\begin{cor}\label{cor:smallerlattice}
  Let $\mcA$ and $\mcB$ be presentations of $M$ with $\mcA\prec \mcB$.
  The sublattice $L_\mcB$ of $L_\mcA$ is a proper sublattice of
  $L_\mcA$ if either of the conditions below holds.
  \begin{enumerate}
  \item There is an $e\in E(M)$ and $h\in s_\mcA(e)$ with
    $s_\mcA(e)-\{h\}\in L_\mcB$ and $s_\mcA(e) \ne s_\mcB(e)$.
  \item For each $I\in 2^{[r]}-L_\mcB$, there is some $h\in I$ with
    $I-\{h\}\in L_\mcB$.
  \end{enumerate}
\end{cor}

\begin{proof}
  Condition (1), Corollary \ref{cor:closedclosetoclosedgen}, and
  Theorem \ref{thm:supportsandclosedsets} give $s_\mcA(e)\in
  L_\mcA-L_\mcB$.  For the second condition, since $\mcA\prec\mcB$,
  there is an $e\in E(M)$ with $s_\mcA(e) \ne s_\mcB(e)$, so condition
  (1) applies.
\end{proof}

\subsection{The intersection of $T_\mcA$ and $T_\mcB$}

We show that, for presentations $\mcA$ and $\mcB$ of a transversal
matroid $M$, the intersection $T_\mcA\cap T_\mcB$ is a sublattice of
$T_\mcA$ and of $T_\mcB$, so for pairs of extensions that are in both
of these lattices, their meet in $T_\mcA$ is their meet in $T_\mcB$,
and likewise for joins.  This line of inquiry is motivated in part by
the following question \cite[Problem 4.1]{extpres}: is the set of all
rank-preserving single-element transversal extensions of a transversal
matroid, ordered by the weak order, a lattice?  An affirmative answer
would provide a transversal counterpart of the following well-known
result of Crapo \cite{crapo}: the set of all single-element extensions
of a matroid $M$, ordered by the weak order, is a lattice.  (This
lattice is called the lattice of extensions of $M$.)  While it is far
from addressing the question about the transversal extensions of a
transversal matroid $M$, the next result, from \cite{extpres}, shows
that the join in $T_\mcA$ is the join in the lattice of extensions of
$M$.

\begin{lemma}\label{lem:easyjoin}
  Let $\mcA$ be a presentation of $M$, and $r=r(M)$.  For any subsets
  $I$ and $J$ of $[r]$, the join of $M[\mcA^I]$ and $M[\mcA^J]$ in the
  lattice of extensions of $M$ is transversal and is $M[\mcA^{I\cup
    J}]$.
\end{lemma}

\begin{cor}
  Let $\mcA$ and $\mcB$ be presentations of a transversal matroid $M$.
  If $M_1$ and $M_2$ are in both $T_\mcA$ and $T_\mcB$, then their
  join in $T_\mcA$ is their join in $T_\mcB$.
\end{cor}

\begin{proof}
  Since $M_1$ and $M_2$ are in both $T_\mcA$ and $T_\mcB$, there are
  sets $I_1$ and $I_2$ in $L_\mcA$, and sets $J_1$ and $J_2$ in
  $L_\mcB$, with $M[\mcA^{I_1}] = M[\mcB^{J_1}] =M_1$ and
  $M[\mcA^{I_2}] = M[\mcB^{J_2}] =M_2$.  By the isomorphism in Theorem
  \ref{thm:faithful2}, the join of $M_1$ and $M_2$ in $T_\mcA$ is
  $M[\mcA^{I_1\cup I_2}]$, and that in $T_\mcB$ is $M[\mcB^{J_1\cup
    J_2}]$.  As claimed, these matroids are equal since, by Lemma
  \ref{lem:easyjoin},
  \begin{equation}\label{eq:join}
    M[\mcA^{I_1\cup I_2}] =  M[\mcA^{I_1}]\join M[\mcA^{I_2}] =
    M[\mcB^{J_1}]\join M[\mcB^{J_2}] = M[\mcB^{J_1\cup J_2}],
  \end{equation} where
  $\join$ denotes the join in the lattice of extensions of $M$.
\end{proof}

The situation for meets is more complex, as the example below
illustrates.

\medskip

\noindent
\textsc{Example 2.}  Consider the matroid $M$ shown in the first two
diagrams in Figure \ref{fig:meetexample}, and the two presentations
given there.  In the extension $M_1=M[\mcA^{\{1\}}]=M[\mcB^{\{1\}}]$,
both $\{x,a,b\}$ and $\{x,c,d\}$ are lines.  In the extension
$M_2=M[\mcA^{\{2\}}]=M[\mcB^{\{2\}}]$, both $\{x,c,d\}$ and
$\{x,e,f\}$ are lines.  In the meet of $M_1$ and $M_2$ in the lattice
of extensions of $M$, each of $\{x,a,b\}$, $\{x,c,d\}$ and $\{x,e,f\}$
is dependent; this meet, which is shown in the third diagram in Figure
\ref{fig:meetexample}, is not transversal.  One way to see this is
that the three coplanar $3$-point lines through $x$ are incompatible
with the affine representation described at the end of Section
\ref{sec:background}.  That view also implies that the meet of $M_1$
and $M_2$ in both $T_\mcA$ and $T_\mcB$ is formed by extending $M$ by
a loop.

\medskip

\begin{figure}
\centering
\begin{tikzpicture}[scale=1]
  \node[inner sep = 0mm] (1) at (0,1) {};
  \node at (-1.45,1) {\small $A_1=\{a,b,c,d,g\}$};
  \node[inner sep = 0mm] (1') at (1.45,1) {};
  \node[inner sep = 0mm] (2) at (1.5,0.2) {};
  \node at (1.5,-0.1) {\small $A_3=\{a,b,e,f,g\}$};
  \node[inner sep = 0mm] (3) at (3,1) {};
  \node at (4.4,1) {\small $A_4=\{g,h\}$};
  \node[inner sep = 0mm] (3') at (1.55,1) {};
  \node[inner sep = 0mm] (4) at (1.5,2.25) {};
  \node at (1.5,2.5) {\small $A_2=\{c,d,e,f,g\}$};
\foreach \from/\to in {1/2,1/1',3'/3,1/4,2/3,2/4,3/4}
\draw(\from)--(\to);

\filldraw (0.5,1.416) node[above left] {$c$} circle  (2pt);
\filldraw (1,1.833) node[above left] {$d$} circle  (2pt);
\filldraw (0.5,0.733) node[below left] {$a$} circle  (2pt);
\filldraw (1,0.466) node[below left] {$b$} circle  (2pt);
\filldraw (1.5,1.5) node[right] {$e$} circle  (2pt);
\filldraw (1.5,0.75) node[right] {$f$} circle  (2pt);
\filldraw (3,1) node[right] {$h$} circle  (2pt);
\filldraw (2.15,1.4) node[below right] {$g$} circle  (2pt);
\end{tikzpicture}

\vspace{15pt}

\begin{tikzpicture}[scale=1]
  \node[inner sep = 0mm] (1) at (0,1) {};
  \node at (-1.45,1) {\small $B_1=\{a,b,c,d,h\}$};
  \node[inner sep = 0mm] (1') at (1.45,1) {};
  \node[inner sep = 0mm] (2) at (1.5,0.2) {};
  \node at (1.5,-0.1) {\small $B_3=\{a,b,e,f,h\}$};
  \node[inner sep = 0mm] (3) at (3,1) {};
  \node at (4.4,1) {\small $B_4=\{g,h\}$};
  \node[inner sep = 0mm] (3') at (1.55,1) {};
  \node[inner sep = 0mm] (4) at (1.5,2.25) {};
  \node at (1.5,2.5) {\small $B_2=\{c,d,e,f,h\}$};
\foreach \from/\to in {1/2,1/1',3'/3,1/4,2/3,2/4,3/4}
\draw(\from)--(\to);

\filldraw (0.5,1.416) node[above left] {$c$} circle  (2pt);
\filldraw (1,1.833) node[above left] {$d$} circle  (2pt);
\filldraw (0.5,0.733) node[below left] {$a$} circle  (2pt);
\filldraw (1,0.466) node[below left] {$b$} circle  (2pt);
\filldraw (1.5,1.5) node[right] {$e$} circle  (2pt);
\filldraw (1.5,0.75) node[right] {$f$} circle  (2pt);
\filldraw (3,1) node[right] {$g$} circle  (2pt);
\filldraw (2.15,1.4) node[below right] {$h$} circle  (2pt);
\end{tikzpicture}

\vspace{15pt}

\begin{tikzpicture}[scale=1]
  \node[inner sep = 0mm] (1) at (0.4,0.5) {};
  \node[inner sep = 0mm] (2) at (0.4,2.5) {};
  \node[inner sep = 0mm] (3) at (1.5,0) {};
  \node[inner sep = 0mm] (4) at (1.5,2) {};
  \node[inner sep = 0mm] (5) at (3.3,0.5) {};
  \node[inner sep = 0mm] (6) at (3.3,2.5) {};

  \node[inner sep = 0mm] (7) at (1.5,1.2) {};
  \node[inner sep = 0mm] (8) at (2.6,0.6) {};
  \node[inner sep = 0mm] (9) at (2.75,1.2) {};
  \node[inner sep = 0mm] (10) at (2.6,2) {};

\foreach \from/\to in {1/2,1/3,2/4,4/3,6/4,3/5,5/6,7/8,7/9,7/10}
\draw(\from)--(\to);

\filldraw (2.2,1.2) node[above] {$c$} circle  (2pt);
\filldraw (2.75,1.2) node[right] {$d$} circle  (2pt);
\filldraw (2,1.56) node[above left] {$a$} circle  (2pt);
\filldraw (2.6,2) node[right] {$b$} circle  (2pt);
\filldraw (2,0.927) node[below left] {$e$} circle  (2pt);
\filldraw (2.6,0.6) node[right] {$f$} circle  (2pt);
\filldraw (0.9,0.8) node[left] {$g$} circle  (2pt);
\filldraw (0.9,1.6) node[left] {$h$} circle  (2pt);
\filldraw (1.5,1.2) node[left] {$x$} circle  (2pt);
\end{tikzpicture}
\caption{The presentations and the meet of the extensions discussed in
  Example 2.  In the first figure, $g$ is in no proper face of the
  simplex; in the second, $h$ is in no proper face.}
\label{fig:meetexample}
\end{figure}
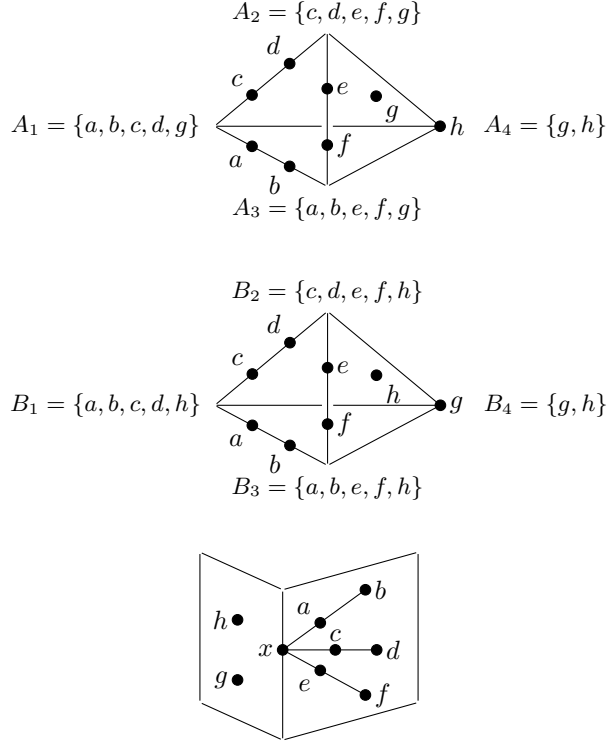

This example illustrates the next result: the meet of $M_1$ and $M_2$
in $T_\mcA$ is their meet in $T_\mcB$ (even though these can differ
from their meet in the lattice of all extensions).

\begin{thm}\label{thm:LABunion}
  If $\mcA$ and $\mcB$ are presentations of $M$, then the set
  $$L_{\mcA,\mcB} = \{I\in L_\mcA\,:\, M[\mcA^I]=M[\mcB^J] \text{ for
    some } J\in L_\mcB\}$$ is a sublattice of $L_\mcA$.  The
  sublattices $L_{\mcA,\mcB}$, of $L_\mcA$, and $L_{\mcB,\mcA}$, of
  $L_\mcB$, are isomorphic, and $T_\mcA\cap T_\mcB$ is a sublattice of
  both $T_\mcA$ and $T_\mcB$.
\end{thm}

The proof of this theorem uses the following result from
\cite{extpres}.

\begin{lemma}\label{lem:supportunion}
  Let $M$ be $M[\mcA]$.  For subsets $X$ and $Y$ of $E(M)$, if
  $r(X)=|s_\mcA(X)|$ and $r(Y)=|s_\mcA(Y)|$, then $r(X\cup
  Y)=|s_\mcA(X\cup Y)|$.
\end{lemma}

\begin{proof}[Proof of Theorem \ref{thm:LABunion}]
  The closure of $L_{\mcA,\mcB}$ under unions follows from the
  argument that gives equation (\ref{eq:join}).  We next show that the
  closure of $L_{\mcA,\mcB}$ under intersections follows from
  statement (\ref{thm:LABunion}.1), which we then prove.
  \begin{enumerate}
  \item[(\ref{thm:LABunion}.1)] \emph{For subsets $X_1,X_2,\ldots,X_t$
      of $E(M)$, if $|s_\mcA(X_k)| = r(X_k) = |s_\mcB(X_k)|$ for all
      $k\in[t]$, then $\bigcap_{k=1}^t s_\mcA(X_k)\in L_{\mcA,\mcB}$.}
  \end{enumerate}
  To see why proving this statement suffices, consider a pair $I_1\in
  L_\mcA$ and $J_1\in L_\mcB$ with $M[\mcA^{I_1}]=M[\mcB^{J_1}]$; let
  $M'$ denote this extension of $M$.  By equation
  (\ref{eq:cirsupint}),
  $$I_1 =\bigcap_{C\in\mcC_x}s_\mcA(C-\{x\})
  \qquad\text{ and } \qquad J_1
  =\bigcap_{C\in\mcC_x}s_\mcB(C-\{x\}),$$ where $\mcC_x$ is the set of
  circuits of $M'$ that contain $x$.  Now
  $s_{\mcA^{I_1}}(C)=s_\mcA(C-\{x\})$ for all $C\in\mcC_x$, so Lemma
  \ref{lem:supp1} gives $|s_\mcA(C-\{x\})|=r(C-\{x\})=|C-\{x\}|$, and
  the corresponding statements hold for $s_\mcB(C-\{x\})$.  The
  corresponding conclusions also hold for any other pair $I_2\in
  L_\mcA$ and $J_2\in L_\mcB$ with $M[\mcA^{I_2}]=M[\mcB^{J_2}]$, so
  $I_1\cap I_2$ has the form $\bigcap_{k=1}^t s_\mcA(X_k)$ that the
  claim treats.

  The case $t=1$ merits special attention: if $|s_\mcA(X)| = r(X) =
  |s_\mcB(X)|$ for some $X\subseteq E(M)$, then $s_\mcA(X)\in
  L_{\mcA,\mcB}$ since $M[\mcA^{s_\mcA(X)}]$ and $M[\mcB^{s_\mcB(X)}]$
  are, by Lemma \ref{lem:princip}, both the principal extension
  $M+_{_X}x$ of $M$.

  Let the sets $X_1,X_2,\ldots,X_t$ be as in statement
  (\ref{thm:LABunion}.1).  Set $I = \bigcap_{k=1}^t s_\mcA(X_k)$ and
  $J = \bigcap_{k=1}^t s_\mcB(X_k)$.  To prove the equality
  $M[\mcA^I]= M[\mcB^J]$, which proves statement
  (\ref{thm:LABunion}.1), by symmetry it suffices to prove that each
  circuit $C$ of $M[\mcA^I]$ that contains $x$ is dependent in
  $M[\mcB^J]$.  Fix such a circuit $C$ of $M[\mcA^I]$.

  We claim that for each $k\in [t]$, we have
  \begin{equation}\label{eqn:Cgoal}
    \bigl|s_\mcA\bigl((C-\{x\})\cup X_k\bigr)\bigr| = r\bigl((C-\{x\})\cup
    X_k\bigr) = \bigl|s_\mcB\bigl((C-\{x\})\cup X_k\bigr)\bigr|.
  \end{equation}  
  To see this, let $\cl$ be the closure operator of $M$, and $\cl_I$
  that of $M[\mcA^I]$.  For any $y\in C-\{x\}$,
  $$\cl\bigl((C-\{x,y\})\cup X_k\bigr) =
  \cl_I\bigl((C-\{x,y\})\cup X_k\bigr)-\{x\}.$$ Lemma \ref{lem:supp1}
  gives $x\in \cl_I(X_k)$.  Thus, $y$ is in
  $\cl_I\bigl((C-\{x,y\})\cup X_k\bigr)$ since $C$ is a circuit of
  $M[\mcA^I]$.  Thus, $y\in \cl\bigl((C-\{x,y\})\cup X_k\bigr)$.  By
  the formulation of closure in terms of circuits (as in
  \cite[Proposition 1.4.11]{oxley}), it follows that each $y \in
  C-(X_k\cup \{x\})$ is in some circuit, say $C_y$, of $M$ with
  $C_y\subseteq X_k\cup (C-\{x\})$. Now
  $|s_\mcA(C_y)|=r(C_y)=|s_\mcB(C_y)|$ by Lemma \ref{lem:supp1}.
  Since this applies for each $y\in C-(X_k\cup \{x\})$, and since we
  also have $|s_\mcA(X_k)| = r(X_k) = |s_\mcB(X_k)|$, equation
  (\ref{eqn:Cgoal}) now follows from Lemma \ref{lem:supportunion}.

  From equation (\ref{eqn:Cgoal}), another application of Lemma
  \ref{lem:supportunion} gives
  $$\Bigl|s_\mcA\Bigl((C-\{x\})\cup \bigl(\bigcup_{k\in P}
  X_k\bigr)\Bigr)\Bigr| =r\Bigl((C-\{x\})\cup \bigl(\bigcup_{k\in P}
  X_k\bigr)\Bigr)= \Bigl|s_\mcB\Bigl((C-\{x\})\cup \bigl(\bigcup_{k\in
    P} X_k\bigr)\Bigr)\Bigr|$$ for any non-empty subset $P$ of $[t]$.
  Thus, for any such $P$,
  $$ \Bigl|\bigcup_{k\in P} s_{\mcA}\bigl((C-\{x\})\cup
  X_k\bigr)\Bigr|= \Bigl|\bigcup_{k\in P} s_{\mcB}\bigl((C-\{x\})\cup
  X_k\bigr)\Bigr|.$$

  Now
  \begin{align*}
    \bigcap_{k=1}^t s_{\mcA}\bigl((C-\{x\})\cup X_k\bigr) = &\,
    \bigcap_{k=1}^t \bigl( s_{\mcA}(C-\{x\})\cup s_{\mcA}(X_k)\bigr) \\
    = &\, s_{\mcA}(C-\{x\})\cup \Bigl( \bigcap_{k=1}^t s_{\mcA}(X_k)\Bigr) \\
    = &\, s_{\mcA}(C-\{x\})\cup  I \\
    = &\, s_{\mcA^I}(C). \\
  \end{align*}
  The same argument applies to $\mcB$ and gives
  $$s_{\mcB^J}(C)=\bigcap_{k=1}^t s_{\mcB}\bigl((C-\{x\})\cup X_k\bigr).$$
 
  The deductions in the previous two paragraphs and
  inclusion-exclusion give
  \begin{align*}
    |s_{\mcA^I}(C)| = & \, \Bigl| \bigcap_{k=1}^t
    s_{\mcA}\bigl((C-\{x\})\cup
    X_k\bigr) \Bigr|\\
    = & \, \sum_{P\subseteq [t]\,:\, P\ne \emptyset} (-1)^{|P|+1} \,
    \Bigl|\bigcup_{k\in P} s_{\mcA}\bigl((C-\{x\})\cup
    X_k\bigr)\Bigr|\\
    = & \, \sum_{P\subseteq [t]\,:\, P\ne \emptyset} (-1)^{|P|+1} \,
    \Bigl|\bigcup_{k\in P} s_{\mcB}\bigl((C-\{x\})\cup
    X_k\bigr)\Bigr|\\
    = & \, \Bigl|\bigcap_{k=1}^t s_{\mcB}\bigl((C-\{x\})\cup
    X_k\bigr)\Bigr|\\
    = & \, |s_{\mcB^J}(C)|.
  \end{align*}
  Since $C$ is a circuit of $M[\mcA^I]$, we have
  $|s_{\mcA^I}(C)|<|C|$.  Thus $|s_{\mcB^J}(C)|<|C|$, so $C$ is
  dependent in $M[\mcB^J]$, as needed.

  The assertions about $L_{\mcB,\mcA}$ and $T_\mcA\cap T_\mcB$ now
  follow easily.
\end{proof}

The proof of Theorem \ref{thm:LABunion} and its reduction to statement
(\ref{thm:LABunion}.1) give the following alternative description of
$L_{\mcA,\mcB}$.

\begin{thm}
  For presentations $\mcA$ and $\mcB$ of $M$, the sublattice
  $L_{\mcA,\mcB}$ of $L_\mcA$ consists of the sets $I\in L_\mcA$ that
  satisfy condition ($*$), as well as all intersections of such sets:
  \begin{quote} 
    \emph{($*$)} \ $I=s_\mcA(X)$ for some $X\subseteq E(M)$ with
    $|s_\mcA(X)|=r(X)=|s_\mcB(X)|$.
  \end{quote}
  The sets $I$ that satisfy condition \emph{($*$)} correspond to the
  principal extensions $M+_{_X}x$ of $M$ that are common to $T_\mcA$
  and $T_\mcB$.
\end{thm}

We conclude this section with two corollaries.  Note that we can
iterate the operation of extending set systems to get
$(\mcA^{I_1})^{I_2}$, where $x_1$ is added in $\mcA^{I_1}$, and $x_2$
is added in $(\mcA^{I_1})^{I_2}$.  We next show that such extensions,
using sets in $L_{\mcA,\mcB}$, are compatible.

\begin{cor}\label{cor:compatible}
  If $M[\mcA^{I_1}]=M[\mcB^{J_1}]$ and $M[\mcA^{I_2}]=M[\mcB^{J_2}]$
  for some sets $I_1,I_2\in L_\mcA$ and $J_1,J_2\in L_\mcB$, then
  $M[(\mcA^{I_1})^{I_2}]=M[(\mcB^{J_1})^{J_2}]$.
\end{cor}

\begin{proof}
  The result follows from two observations: (i) Theorem
  \ref{thm:chraracerizeclosedsets} yields $I_2\in L_{\mcA^{I_1}}$ and
  $J_2\in L_{\mcB^{J_1}}$; (ii) if $I_2$ and $X$ satisfy condition
  ($*$) above in $M$, then so do $I_2$ and $X$ in $M[\mcA^{I_1}]$, and
  likewise for intersections of sets that satisfy condition ($*$).
\end{proof}

\begin{cor}
  For $I\in L_\mcA$ and $J\in L_\mcB$, if $M[\mcA^I] = M[\mcB^J]$,
  then $|I|=|J|$.
\end{cor}

\begin{proof}
  Apply Corollary \ref{cor:compatible} repeatedly, with each $I_h=I$
  and each $J_h=J$, until the set of added elements is cyclic in the
  extension; the rank of this cyclic set must be both $|I|$ and $|J|$.
\end{proof}

\subsection{How to get any finite distributive lattice}

We show that each sublattice of $2^{[r]}$ that includes both
$\emptyset$ and $[r]$ is the lattice $L_\mcA$ for some presentation
$\mcA$ of some transversal matroid of rank $r$; indeed, we prove two
refinements of this result.  Up to isomorphism, this result covers all
finite distributive lattices since each such lattice $L$ is isomorphic
to the lattice of order ideals of some finite ordered set
(specifically, the induced order on the set of join-irreducible
elements of $L$; see, e.g., \cite[Theorem II.2.5]{aigner}).  Combining
the result below with Theorem \ref{thm:faithful2} shows any
distributive lattice is isomorphic to $T_\mcA$ for some presentation
$\mcA$ of some transversal matroid.

\begin{thm}\label{thm:alldistributivelattices}
  Let $L$ be a sublattice of $2^{[r]}$ that contains both $\emptyset$
  and $[r]$.
  \begin{enumerate}
  \item There is a rank-$r$ transversal matroid $M$ and maximal
    presentation $\mcA$ of $M$ with $L=L_\mcA$.
  \item For any $n\geq r$, there is a presentation $\mcB$ of the
    uniform matroid $U_{r,n}$ with $L=L_\mcB$.
  \end{enumerate}
\end{thm}

\begin{proof}
  To prove assertion (1), for each non-empty set $I\in L$, let $X_I$
  be a set of $|I|+1$ elements that is disjoint from all other such
  sets $X_J$.  For $i$ with $1\leq i\leq r$, let
  $$A_i = \bigcup_{I\in L\,:\,i\in I}X_I,$$ so the elements of $X_I$
  are in exactly $|I|$ of the sets $A_i$ (counting multiplicity; we
  may have $A_i=A_j$ even if $i\ne j$).  Let $\mcA = (A_i\,:\,i\in
  [r])$ and let $M$ be the matroid $M[\mcA]$ on
  $$E(M) =  \bigcup_{I\in L\,:\,I\ne\emptyset}X_I = \bigcup_{i=1}^r A_i.$$
  Thus, if $e\in X_I$, then $s_\mcA(e)=I$.  The presentation $\mcA$ of
  $M$ is maximal since, with $|X_I|>|I|$ and $s_{\mcA}(X_I)=I$, the
  set $X_I$ is dependent in $M$, yet if we adjoin any element of $X_I$
  to any set $A_j$ with $j\not\in I$, then the resulting set system
  $\mcA'$ has a matching of $X_I$, so $X_I$ is independent in
  $M[\mcA']$.  It now follows from Theorem
  \ref{thm:supportsandclosedsets} that $L\subseteq L_\mcA$.  Since $L$
  and $L_\mcA$ are sublattices of $2^{[r]}$ and $s_{\mcA}(e)\in L$ for
  all $e\in E(M)$ by construction, we get $s_{\mcA}(F)\in L$ for each
  cyclic flat $F$ of $M$, so Corollary \ref{cor:cyclicgivesclosed}
  gives $L_\mcA\subseteq L$.  Thus, $L_\mcA= L$.

  Figure \ref{fig:latticeconstruction} illustrates the proof of
  assertion (2).  Let $[n]$ be the ground set of $U_{r,n}$.  For $I\in
  L$, let $I_0$ be the (possibly empty) set of elements that occur
  first in $I$, that is, $$I_0 = I-\bigcup_{J\in L\,:\,J\subsetneq I}
  J.$$ Since $L$ is closed under intersection, for each $i\in [r]$,
  there is exactly one $I\in L$ with $i\in I_0$; using that $I$,
  set $$B_i =([n]-[r])\cup \bigcup_{J\in L\,:\,I\subseteq J}J_0.$$ By
  construction, $|\mcB|=r$ and $i\in B_i$, so $[r]$ is a basis of
  $M[\mcB]$.  Since $[n]-[r]\subseteq B_i$ for all $i\in[r]$, it
  follows that $M[\mcB]$ is the uniform matroid $U_{r,n}$.  For $i\in
  I_0$ and $j\in J_0$, we have $i \in B_j$ if and only if $J\subseteq
  I$, so $s_\mcB(i)=I$.  Since $L$ is closed under unions, we get
  $s_\mcB(X)\in L$ for all $X\subseteq [r]$.  Also, each set $I\in L$
  is independent in $U_{r,n}$ and $s_\mcB(I)=I$. From these
  observations and Theorem \ref{thm:chraracerizeclosedsets}, we get
  $L= L_\mcB$.
\end{proof}

\begin{figure}
  \centering
  \begin{tikzpicture}[scale=0.8]
  \node[inner sep = 0.3mm] (em) at (1,0) {\small $\emptyset$};
  \node[inner sep = 0.3mm]  (1) at (1,1) {\small $\{1\}$};
  \node[inner sep = 0.3mm] (123) at (0,2) {\small $\{1,2,3\}$};
  \node[inner sep = 0.3mm] (145) at (2,2) {\small $\{1,4,5\}$};
  \node[inner sep = 0.3mm] (15) at (1,3) {\small $\{1,2,3,4,5\}$};
  \node[inner sep = 0.3mm] (16) at (1,4) {\small $\{1,2, 3, 4, 5, 6\}$};
  \foreach \from/\to in {em/1,1/123,1/145,123/15,145/15,15/16}
  \draw(\from)--(\to);

  \node[inner sep = 0.3mm] (xem) at (6,0) {\small $\emptyset$};
  \node[inner sep = 0.3mm]  (x1) at (6,1) {\small $\{1\}$};
  \node[inner sep = 0.3mm] (x123) at (5,2) {\small $\{2,3\}$};
  \node[inner sep = 0.3mm] (x145) at (7,2) {\small $\{4,5\}$};
  \node[inner sep = 0.3mm] (x15) at (6,3) {\small $\emptyset$};
  \node[inner sep = 0.3mm] (x16) at (6,4) {\small $\{6\}$};
  \foreach \from/\to in {xem/x1,x1/x123,x1/x145,x123/x15,x145/x15,x15/x16}
  \draw(\from)--(\to);

  \end{tikzpicture}
  \caption{An example, for $U_{6,7}$, of the construction 
    of $\mcB$ in the proof of Theorem
    \ref{thm:alldistributivelattices}, with 
     $L$ on the left and the sets $I_0$ on the right.
     The presentation has $B_1=\{1,2,3,4,5,6,7\}$,
     $B_2=B_3=\{2,3,6,7\},$ \ $B_4=B_5=\{4,5,6,7\}$, \ and \
     $B_6=\{6,7\}$.}
  \label{fig:latticeconstruction}
  \end{figure}
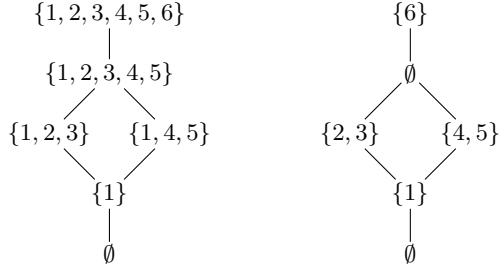

\subsection{Irreducible elements}

An element $a$ in a lattice $L$ is \emph{join-irreducible} if (i) $a$
is not the least element of $L$ and (ii) if $a=b\join c$, then
$a\in\{b,c\}$.  Dually, $a$ is \emph{meet-irreducible} if (i$'$) $a$
is not the greatest element of $L$ and (ii$'$) if $a=b\meet c$, then
$a\in\{b,c\}$.  (While not all authors include them, conditions (i)
and (i$'$) shorten the wording of results.)

The irreducible elements of a finite distributive lattice $L$ are of
great interest.  The order induced on the set of join-irreducibles of
$L$ is isomorphic to that induced on its set of meet-irreducibles, and
the lattice of order ideals of each of these induced suborders of $L$
is isomorphic to $L$ itself.  (See, e.g., \cite[Theorem II.2.5 and
Corollary II.2.7]{aigner}.)  Thus, the rank of $L$ is the number of
join-irreducibles in $L$, which is also its number of
meet-irreducibles.

We now study the irreducible elements of the lattices $L_\mcA$
introduced above.

The least set $S_i$ in $L_\mcA$ that contains a given element $i\in
[r]$ is $\bigcap_{J\in L_\mcA\,:\, i\in J}J$. The sets $S_i$ are not
limited to the atoms of $L_\mcA$; see the examples in Figure
\ref{fig:example}.  Clearly $S_i$ is join-irreducible. Each set $U$ in
$L_\mcA$ is $\bigcup_{i\in U}S_i$, so there are no other
join-irreducibles of $L_\mcA$.  Thus, the number of join-irreducibles
is the number of distinct sets $S_i$.  Note that if $A_i$ and $A_j$ in
$\mcA$ are equal, then $S_i=S_j$ since, for $X\subseteq E(M)$, we have
$i\in s_\mcA(X)$ if and only if $j\in s_\mcA(X)$.  Thus, the number of
join-irreducible sets in $L_\mcA$ is at most the number of distinct
sets in $\mcA$.  As Example 1 shows, this bound can be strict (there,
$\mcA$ has three distinct sets but $L_\mcA$ has only one
join-irreducible; likewise for $\mcB$).  

The greatest set in $L_\mcA$ that does not contain a given element
$i\in [r]$ is $\bigcup_{J\in L_\mcA\,:\, i\not\in J}J$.  An argument
like that above, or an application of order-duality, shows that these
are the meet-irreducibles of $L_\mcA$.  By the remark after the proof
of Theorem \ref{thm:chraracerizeclosedsets}, each meet-irreducible
element of $L_\mcA$ corresponds to a principal extension of $M$; the
converse is false, since for instance, in either example in Figure
\ref{fig:example}, the set $\{2,3\}$ corresponds to a principal
extension, but $\{2,3\}$ is the meet of the sets $\{1,2,3\}$ and
$\{2,3,4\}$ in $L_\mcA$.

We now identify a join-sublattice $L'_\mcA$ of $L_\mcA$ that, by
Theorem \ref{thm:chraracerizeclosedsets}, has the same the
meet-irreducibles, thereby reducing the problem of finding the
meet-irreducibles of $L_\mcA$ to the same problem on a potentially
smaller lattice.  Set
$$L'_\mcA=\{s_\mcA(X)\,:\,X\subseteq E(M),\, |s_\mcA(X)| = r(X)\}.$$
(Adding the condition that $X$ is independent would not change
$L'_\mcA$.)  By Theorem \ref{thm:chraracerizeclosedsets},
$L'_\mcA\subseteq L_\mcA$ and $L'_\mcA$ generates $L_\mcA$ since
$L_\mcA$ consists precisely of the intersections of the sets in
$L'_\mcA$.  Lemma \ref{lem:supportunion} shows that $L'_\mcA$ is a
join-sublattice of $L_\mcA$.

Each lattice is isomorphic to $L'_\mcA$
for a maximal presentation $\mcA$ of some transversal matroid (see the
proof of \cite[Theorem 2.1]{cyclic}).  By Corollary
\ref{cor:cyclicgivesclosed}, when the presentation $\mcA$ is maximal,
the same conclusions hold for the (often smaller) lattice
$$L''_\mcA=\{s_\mcA(X)\,:\,X\text{ is a cyclic flat of } M\}\cup [r].$$

\section{Applications}\label{sec:applic}

Theorems \ref{thm:threequarters} and \ref{thm:intbd} below are
applications of the results in Section \ref{sec:lattice}.  Both
results stem from the observation that proper sublattices of $2^{[r]}$
must be substantially smaller than $2^{[r]}$.  (The special case of
maximal proper sublattices of $2^{[r]}$ have been studied in other
settings, such as finite topologies; see, e.g., Sharp \cite{sharp} and
Stephen \cite{stephen}.)

\begin{thm}\label{thm:threequarters}
  Let $M$ be a transversal matroid of rank $r$, and let $\mcA^i$ be a
  presentation of $M$ that has rank $i$ in the ordered set of
  presentations of $M$.  If $1\leq i<r$, then
  $$|T_{\mcA^i}|=|L_{\mcA^i}| 
  \leq \bigl(\frac{1}{2}+ \frac{1}{2^{i+1}}\bigr) 2^r;$$ these bounds
  are sharp. Also, if $i\geq r$, then $|T_{\mcA^i}| =|L_{\mcA^i}| \leq
  2^{r-1}$.
\end{thm}

We first give examples to show that, for $1\leq i<r$, the bounds are
sharp.  (These examples, which play a role in the proof of the bound,
have coloops; to get examples without coloops, take free extensions of
these.)  Let $\mcB =(B_2,B_3,\ldots,B_r)$ be a minimal presentation of
a transversal matroid $N$ of rank $r-1$.  Fix an element $e\not\in
E(M)$ and let $M$ be the direct sum of $N$ and the rank-$1$ matroid on
$\{e\}$.  For $0\leq k< r$, define $\mcA^k = (A^k_i\,:i\in [r])$ by
$$A^k_i=\left\{
  \begin{array}{ll}
    \{e\}, &\mbox{ if } i=1,\\
    B_i\cup \{e\}, &\mbox{ if } 2\leq i\leq k+1,\\ 
    B_i, &\mbox{ otherwise. }\\ 
  \end{array}
\right.$$ Thus, $s_{\mcA^k}(e) = [k+1]$.  Each $\mcA^k$ is a
presentation of $M$, the presentation $\mcA^0$ is minimal, and
$\mcA^{k-1}\cov\mcA^k$ for $k\geq 1$.  Thus, $\mcA^k$ has rank $k$ in
the ordered set of presentations.  Since $\mcB$ is a minimal
presentation of $N$, each subset of $\{2,3,\ldots,r\}$ is in
$L_{\mcA^k}$.  Thus, since $s_{\mcA^k}(e) = [k+1]$, Corollary
\ref{cor:closedclosetoclosedgen} implies that all supersets of $[k+1]$
are in $L_{\mcA^k}$.  Since $1\in s_{\mcA^k}(X)$ if and only if $e\in
X$, by Theorem \ref{thm:chraracerizeclosedsets} the sets in
$L_{\mcA^k}$ that contain $1$ must contain all of $[k+1]$.  Thus,
$L_{\mcA^k}$ consists of the subsets of $[r]$ that either do not
contain $1$ or contain all of $[k+1]$.  For reasons that Lemma
\ref{lem:classifylattices} will reveal, it is useful to recast this as
follows: $L_{\mcA^k}$ is the complement, in $2^{[r]}$, of the union of
the intervals $$[\{1\},\overline{\{2\}}],\,\,\,
[\{1,2\},\overline{\{3\}}],\,\,\,
[\{1,2,3\},\overline{\{4\}}],\,\,\,\ldots,\,\,\,
[\{1,2,\ldots,k\},\overline{\{k+1\}}],$$ where $\overline{X}$ denotes
the complement of the set $X$.  From the first description of
$L_{\mcA^k}$, we get $$|L_{\mcA^k}| = 2^{r-1}+2^{r-(k+1)} =
\bigl(\frac{1}{2}+ \frac{1}{2^{k+1}}\bigr) 2^r.$$

The proof of the bound in Theorem \ref{thm:threequarters} uses Lemma
\ref{lem:classifylattices}, which catalogs the sublattices of
$2^{[r]}$ that have more than $2^{r-1}$ elements.  The proof of that
lemma uses the following result by Chen, Koh, and Tan \cite{ckt} (see
the proof in Rival \cite{rival}).

\begin{lemma}\label{lem:ckt}
  Let $\mathcal{J}$ be the set of join-irreducibles of a finite
  distributive lattice $L$, and $\mathcal{M}$ its set of
  meet-irreducibles.  The maximal proper sublattices of $L$ are
  precisely the differences $L-[a,b]$ where the interval $[a,b]$ in
  $L$ satisfies $[a,b]\cap \mathcal{J} = \{a\}$ and $[a,b]\cap
  \mathcal{M} = \{b\}$.
\end{lemma}

\begin{lemma}\label{lem:classifylattices}
  Up to permutations of $[r]$, the sublattices of $2^{[r]}$ that have
  more than $2^{r-1}$ elements are $L_i = 2^{[r]}-U_i$ and
  $L'_i=2^{[r]}-U'_i$, for $1\leq i<r$, where
  $$U_i = \bigcup_{j\,:\, 1\leq j\leq i}  [\{1,2,\ldots,j\},\overline{\{j+1\}}]
  \quad \text{ and } \quad U'_i = \bigcup_{j\,:\,1\leq j\leq i}
  [\{j+1\},\overline{\{1,2,\ldots,j\}}],$$ and $L_V=2^{[r]}-V$ where
  $V \,= \,[\{1\},\overline{\{2\}}]\,\cup\, [\{3\},\overline{\{4\}}]$.
  Thus, $|L_i|=|L'_i|= \bigl(\frac{1}{2}+ \frac{1}{2^{i+1}}\bigr) 2^r$
  and $|L_V|=\frac{9}{16}\cdot 2^r$.  Also, $L_V$ is not contained in
  any sublattice $L$ of $2^{[r]}$ with $|L|=\frac{5}{8}\cdot 2^r$.
\end{lemma}

\begin{proof}
  To prove this result, we apply Lemma \ref{lem:ckt} recursively.  To
  simplify the argument, note that $U'_i$ is the image of $U_i$ under
  the complementation map $X\mapsto \overline{X}$ (which is
  order-reversing) of $2^{[r]}$; this allows us to pursue only the
  lattices $L_V$ and $L_1,L_2,\ldots,L_{r-1}$ below.
 
  The join-irreducibles of $2^{[r]}$ are the singleton sets, and the
  meet-irreducibles are their complements, so by Lemma \ref{lem:ckt},
  the maximal proper sublattices of $2^{[r]}$ are $L_1$ and its images
  under permutations of $[r]$ (the lattice $L'_1$ is obtained by such
  a permutation).

  To verify the assertions below about join-irreducibles, note that
  (i) each join-irreducible of $L_{i-1}$ that is also in $L_i$ is
  join-irreducible in $L_i$, and (ii) $L_i$ has at most $r$
  join-irreducibles.  (The second statement holds since the rank of a
  distributive lattice is its number of join-irreducibles; see
  \cite[Corollary II.2.11]{aigner}.)  Similar observations apply to
  meet-irreducibles.

  We now find the maximal proper sublattices of $L_1 = 2^{[r]} -
  [\{1\},\overline{\{2\}}]$.  Its join-irreducibles are $\{i\}$, for
  $2\leq i\leq r$, along with $\{1,2\}$; its meet-irreducibles are
  $\overline{\{i\}}$, for $i\in [r]-\{2\}$, along with
  $\overline{\{1,2\}}$.  Up to the map $X\mapsto \overline{X}$ (which
  maps $L_2$ to $L'_2$) and permuting $3, 4, \ldots, r$, there are
  three maximal proper sublattices, namely
  \begin{enumerate}
  \item[(1)] $L_2=L_1 - [\{1,2\},\overline{\{3\}}]$, which has
    $\frac{5}{8}\cdot 2^r$ elements,
  \item[(2)] $L_V=L_1-[\{3\},\overline{\{4\}}]$, which has
    $\frac{9}{16}\cdot 2^r$ elements, and
  \item[(3)] $L_1-[\{2\},\overline{\{1\}}]$, which has $2^{r-1}$
    elements.
  \end{enumerate}
  (The join-irreducible $\{1,2\}$ is in $[\{2\},\overline{\{3\}}]$, so
  this interval is not listed.  Likewise for $\overline{\{1,2\}}$ and
  $[\{3\},\overline{\{1\}}]$.)  Only $L_2$ and $L_V$ are of interest
  for the lemma.

  The join-irreducibles of $L_V$ are $\{i\}$, for $i\in [r]-\{1,3\}$,
  along with $\{1,2\}$ and $\{3,4\}$; its meet-irreducibles are
  $\overline{\{j\}}$, for $j\in [r]-\{2,4\}$, along with
  $\overline{\{1,2\}}$ and $\overline{\{3,4\}}$.  Up to switching the
  pair $(1,2)$ with the pair $(3,4)$, permuting $5, 6, \ldots, r$, and
  the map $X\mapsto \overline{X}$, there are three maximal proper
  sublattices of $L_V$ (omitting the case covered by (3) above):
  \begin{enumerate}
  \item[(4)] $L_V - [\{1,2\},\overline{\{3,4\}}]$, which has $2^{r-1}$
    elements,
  \item[(5)] $L_V-[\{1,2\},\overline{\{5\}}]$, which has
    $\frac{15}{32}\cdot 2^r$ elements, and
  \item[(6)] $L_V-[\{5\},\overline{\{6\}}]$, which has
    $\frac{27}{64}\cdot 2^r$ elements.
  \end{enumerate}
  Thus, no proper sublattices of $L_V$ have more than $2^{r-1}$
  elements.

  To complete the proof, we induct to show that for $i$ with $3\leq i<
  r$, the only maximal proper sublattice $L$ of $L_{i-1}$ with
  $|L|>2^{r-1}$ is $L_i$, up to permuting elements.  We include the
  following conditions in the induction argument (see Figure
  \ref{fig:irreducibles}):
  \begin{enumerate}
  \item[(i)] the join-irreducibles of $L_{i-1}$ are $\{j\}$, for $1<
    j\leq r$, along with $[i]$, and
  \item[(ii)] the meet-irreducibles of $L_{i-1}$ are
    $\overline{\{1\}}$ and $\overline{\{k\}}$, for $i<k\leq r$, along
    with $\overline{\{1,t\}}$ where $2\leq t\leq i$.
  \end{enumerate}
  Conditions (i) and (ii) are easy to see in the base case, $i=3$.  We
  use the same argument for the base case as for the inductive step.
  Let $L$ be a maximal proper sublattice of $L_{i-1}$.  If $L=L_{i-1}-
  [A,B]$ where $|A|=1$ and $B=\overline{\{1,t\}}$ with $2\leq t\leq
  i$, then $[A,B]$ is disjoint from $U_{i-1}$ and has $2^{r-3}$
  elements, so $|L| \leq 2^{r-1}$.  If $L = L_{i-1} -
  [\{j\},\overline{\{k\}}]$, with $j$ and $k$ distinct elements of
  $\{i+1, i+2, \ldots, r\}$, then $|L|\leq \frac{15}{32}\cdot 2^r$ by
  case (5) (with relabelling).  Thus, up to relabelling, only $L_i
  =L_{i-1} - [\{1,2,\ldots,i\},\overline{\{i+1\}}]$ has more than
  $2^{r-1}$ elements: $|L_i| = \bigl(\frac{1}{2}+
  \frac{1}{2^{i+1}}\bigr) 2^r$ .  It is easy to check that conditions
  (i) and (ii) hold for $L_i$, which completes the induction.
\end{proof}

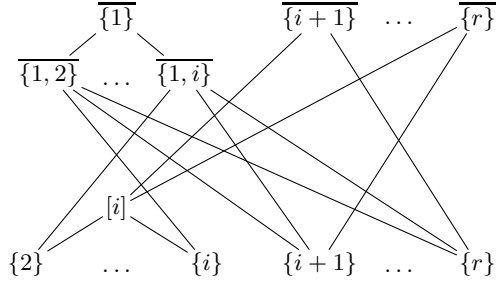
\begin{figure}
  \centering
  \begin{tikzpicture}[scale=0.6]
  \node[inner sep = 0.3mm] (2) at (-1,-0.5) {\small $\{2\}$};
  \node[inner sep = 0.3mm] (d) at (1,-0.65) {\small $\ldots$};
  \node[inner sep = 0.3mm]  (i-) at (3,-0.5) {\small $\{i\}$};
  \node[inner sep = 0.3mm]  (i+1-) at (5.5,-0.5) {\small $\{i+1\}$};
  \node[inner sep = 0.3mm] (d+) at (7.25,-0.65) {\small $\ldots$};
  \node[inner sep = 0.3mm]  (r-) at (9,-0.5) {\small $\{r\}$};
  \node[inner sep = 0.3mm] (i) at (1,0.75) {\small $[i]$};
  \node[inner sep = 0.3mm] (1i) at (2.5,3.75) {\small $\overline{\{1,i\}}$};
  \node[inner sep = 0.3mm] (dd) at (1,3.5) {\small $\ldots$};
  \node[inner sep = 0.3mm] (12) at (-0.5,3.75) {\small $\overline{\{1,2\}}$};
  \node[inner sep = 0.3mm] (1) at (1,5) {\small $\overline{\{1\}}$};
  \node[inner sep = 0.3mm]  (i+1t) at (5.5,5) {\small $\overline{\{i+1\}}$};
  \node[inner sep = 0.3mm] (d+t) at (7.25,4.85) {\small $\ldots$};
  \node[inner sep = 0.3mm]  (rt) at (9,5) {\small $\overline{\{r\}}$};
  \foreach \from/\to in
  {1/12,1/1i,i/2,i/i-,12/i-,1i/2,i/i+1t,i/rt,
    i+1-/12,i+1-/1i,r-/12,r-/1i,rt/i+1-,i+1t/r-}   
  \draw(\from)--(\to);
  \end{tikzpicture}
  \caption{The induced order on the irreducibles of $L_{i-1}$.}
  \label{fig:irreducibles}
\end{figure}

The last background item we need before proving the upper bounds in
Theorem \ref{thm:threequarters} is the following lemma from
\cite{extpres}.

\begin{lemma}\label{lem:samenumber}
  Let $\mcA$ be a presentation of $M$.  Fix $Y\subseteq E(M)$.  If
  $r(M\del Y) = r(M)$, then $M$ has a minimal presentation $\mcC$ with
  $\mcC\preceq \mcA$ so that $s_\mcC(e)= s_\mcA(e)$ for all $e\in Y$.
\end{lemma}

\begin{proof}[Proof of Theorem \ref{thm:threequarters}.]
  Consider presentations $\mcA^0\cov\mcA^1\cov \cdots \cov \mcA^r$ of
  $M$ where $\mcA^0$ is minimal.  Thus, $\mcA^j$ has rank $j$ in the
  order on presentations, and $L_{\mcA^j}$ is a sublattice of
  $L_{\mcA^{j-1}}$.  By Lemma \ref{lem:classifylattices}, if
  $|L_{\mcA^j}| >2^{r-1}$, then $|L_{\mcA^j}|=\bigl(\frac{1}{2}+
  \frac{1}{2^{i+1}}\bigr) 2^r$ for some $i$ with $1\leq i<r$, so it
  suffices to prove the following statement:
  \begin{quote}
    \emph{if $\displaystyle{|L_{\mcA^j}| = \bigl(\frac{1}{2}+
      \frac{1}{2^{i+1}}\bigr) 2^r}$, then $j\leq i$.}
  \end{quote}

  For $i=1$, assume $|L_{\mcA^j}| = \frac{3}{4}\cdot 2^r$.  By Lemma
  \ref{lem:classifylattices}, up to permuting $[r]$, we have
  $L_{\mcA^j} = 2^{[r]} - [\{1\},\overline{\{2\}}]$.  Condition (2) of
  Corollary \ref{cor:smallerlattice} holds ($h$ is $1$), so
  $L_{\mcA^j}$ is properly contained in $L_{\mcA^{j-1}}$; since
  $L_{\mcA^j}$ is a proper sublattice only of $2^{[r]}$, we have
  $L_{\mcA^{j-1}} = 2^{[r]}$.  Thus, $\mcA^{j-1}$ is a minimal
  presentation by Theorem \ref{thm:charmin}, so $j-1=0$, so $j=1$.

  For $i=2$, if $|L_{\mcA^j}| = \frac{5}{8}\cdot 2^r$, then, by Lemma
  \ref{lem:classifylattices}, up to permuting $[r]$, the lattice
  $L_{\mcA^j}$ is either
  $$2^{[r]} -\bigl( [\{1\},\overline{\{2\}}] \cup
  [\{1,2\},\overline{\{3\}}]\bigr) \qquad \text{ or } \qquad 2^{[r]} -
  \bigl( [\{2\},\overline{\{1\}}] \cup
  [\{3\},\overline{\{1,2\}}]\bigr).$$ Condition (2) of Corollary
  \ref{cor:smallerlattice} holds ($h$ is $1$ in the first case and
  either $2$ or $3$ in the second), so $L_{\mcA^j}$ is properly
  contained in $L_{\mcA^{j-1}}$.  Thus, $|L_{\mcA^{j-1}}|\geq
  \frac{3}{4}\cdot 2^r$.  The previous case gives $j-1\leq 1$, so
  $j\leq 2$.
  
  The general case with $L_{\mcA^j} = L_i$ or $L_{\mcA^j} = L'_i$
  follows inductively in the same manner.  We turn to the only case
  that requires a more involved argument, namely $$L_{\mcA^j} = L_V =
  2^{[r]} - \bigl(\,[\{1\},\overline{\{2\}}]\,\cup\,
  [\{3\},\overline{\{4\}}]\,\bigr).$$ Since $\mcA^{j-1}\cov \mcA^j$,
  we have $s_{\mcA^{j-1}}(e) \subsetneq s_{\mcA^j}(e)$ for some $e\in
  E(M)$, so $s_{\mcA^{j-1}}(e) \not \in L_V$ by Theorem
  \ref{thm:supportsandclosedsets}. Thus, $s_{\mcA^{j-1}}(e) \in
  [\{1\},\overline{\{2\}}]\,\cup\, [\{3\},\overline{\{4\}}]$.  If
  $s_{\mcA^{j-1}}(e)$ is in only one of $[\{1\},\overline{\{2\}}]$ and
  $[\{3\},\overline{\{4\}}]$, then $L_{\mcA^j}$ is a proper sublattice
  of $L_{\mcA^{j-1}}$ by condition (1) of Corollary
  \ref{cor:smallerlattice}; thus, $|L_{\mcA^{j-1}}|\geq
  \frac{3}{4}\cdot 2^r$, so $j-1\leq 1$, so $j< 3$.  We may now assume
  that $L_{\mcA^j}= L_{\mcA^{j-1}}$ and that $s_{\mcA^{j-1}}(e) \in
  [\{1\},\overline{\{2\}}]\,\cap\, [\{3\},\overline{\{4\}}]$.

  First assume that for all options for the terms
  $\mcA^0,\mcA^1,\ldots,\mcA^{j-1}$, the only element $d$ with
  $s_{\mcA^j}(d)\ne s_{\mcA^k}(d)$ for some $k<j$ is $d=e$.  Lemma
  \ref{lem:samenumber} then implies that $e$ is a coloop of $M$; also,
  the presentation of $M\del e$ that is obtained by removing $e$ from
  all sets in $\mcA^0$ is minimal.  This case is covered by the
  example that we used to show that the bound is sharp, so we may now
  assume that $e$ is not a coloop of $M$.

  In this case, by Lemma \ref{lem:samenumber} with $J = \{e\}$, we can
  choose $\mcA^0,\mcA^1,\ldots,\mcA^{j-2}$ so that $s_{\mcA^{j-1}}(e)
  =s_{\mcA^{j-2}}(e)$.  Since $\mcA^{j-2}\cov\mcA^{j-1}$, we have
  $s_{\mcA^{j-2}}(e') \subsetneq s_{\mcA^{j-1}}(e')$ for some $e'\in
  E(M)$.  Thus, $e'\ne e$.  Now $s_{\mcA^{j-2}}(e') \not \in L_V$ by
  Theorem \ref{thm:supportsandclosedsets}, so $s_{\mcA^{j-2}}(e')$ is
  in either $[\{1\},\overline{\{2\}}]$ or $[\{3\},\overline{\{4\}}]$.
  If $s_{\mcA^{j-2}}(e')$ is not in both intervals, then the argument
  above gives the result, so assume $s_{\mcA^{j-2}}(e')\in
  [\{1\},\overline{\{2\}}]\,\cap\, [\{3\},\overline{\{4\}}]$.  Set $F
  = \{e,e'\}$.  Thus,
  $$s_{\mcA^{j-2}}(F) = s_{\mcA^{j-2}}(e)\cup s_{\mcA^{j-2}}(e') \in
  [\{1\},\overline{\{2\}}]\,\cap\, [\{3\},\overline{\{4\}}].$$
  Corollary \ref{cor:closedclosetoclosedgen} with $J=
  s_{\mcA^{j-2}}(F)-\{1,3\}$, and so $H = \{1,3\}$, gives
  $s_{\mcA^{j-2}}(F) \in L_{\mcA^{j-2}}$, so $L_{\mcA^j}$ is a proper
  sublattice of $L_{\mcA^{j-2}}$.  Lemma \ref{lem:classifylattices}
  gives $|L_{\mcA^{j-2}}|\geq \frac{3}{4}\cdot 2^r$; thus, $j-2\leq
  1$, so $j\leq 3$, as needed.
\end{proof}

Let $\mcA$ and $\mcB$ be presentations of $M$.  In Theorem
\ref{thm:LABunion} we showed that $T_\mcA\cap T_\mcB$ is a sublattice
of both $T_\mcA$ and $T_\mcB$.  The smallest that $|T_\mcA\cap
T_\mcB|$ can be is two, with these two common extensions being the
free extension and the extension by a loop; for instance, the two
minimal presentations
$$\mcA=(\{i\}\cup ([2r]-[r])\,:\, i\in [r])\quad
\text{ and } \quad \mcB=([r]\cup \{i\}\,:\, i\in [2r]-[r])$$ of
$U_{r,2r}$ on $[2r]$ have this property.  We conclude with a sharp
upper bound on $|T_\mcA\cap T_\mcB|$.

\begin{thm}\label{thm:intbd}
  If the presentations $\mcA=(A_i\,:\,i\in [r])$ and
  $\mcB=(B_i\,:\,i\in [r])$ of $M$ differ by more than just reindexing
  the sets, then $|T_\mcA\cap T_\mcB|\leq \frac{3}{4}\cdot 2^r$.  This
  bound is sharp.
\end{thm}

\begin{proof}
  The inequality follows from Theorems \ref{thm:threequarters} and
  \ref{thm:LABunion} if either $\mcA$ or $\mcB$ is not minimal, so we
  may assume that both are minimal.  As shown in Section \ref{sec:T},
  when $\mcA$ is minimal, we can reconstruct the sets in $\mcA$ from
  $T_\mcA$; thus, by our assumption, $T_\mcA\ne T_\mcB$, so
  $L_{\mcA,\mcB}$ is a proper sublattice of $L_\mcA$.  Thus, we get
  the bound by our work above.

  To see that this bound is tight, let $M$ be $U_{r-2,r-2} \oplus
  U_{2,3}$, with $U_{r-2,r-2}$ and $U_{2,3}$ on the sets
  $\{e_1,e_2,\ldots,e_{r-2}\}$ and $\{e_{r-1},a,b\}$, respectively.
  Consider the presentations $\mcA=(A_i\,:\,i\in [r])$ and
  $\mcB=(B_i\,:\,i\in [r])$ where $A_i = B_i = \{e_i\}$ for $i\in
  [r-2]$ and
  $$A_{r-1}=\{e_{r-1},a\}, \qquad B_{r-1}=\{e_{r-1},b\}, \qquad
  A_r=B_r = \{a,b\}.$$ By Lemma \ref{lem:princip}, if $I\subseteq
  [r-1]$, then both $M[\mcA^I]$ and $M[\mcB^I]$ are the principal
  extension $M+_Y x$ where $Y = \{e_i\,:\,i\in I\}$; also, if
  $\{r-1,r\}\subseteq I\subseteq [r]$, then $M[\mcA^I]$ and
  $M[\mcB^I]$ are both $M+_Y x$ where $Y = \bigl\{e_i\,:\,i\in
  I-\{r\}\bigr\}\cup \{a,b\}$.  There are
  $2^{r-1}+2^{r-2}=\frac{3}{4}\cdot 2^r$ such sets $I$, so the bound
  is optimal.
\end{proof}

\bigskip

\begin{center}
  \textsc{Acknowledgments}
\end{center}

The author thanks Anna de Mier for very useful feedback on the ideas
in this paper, for comments that improved the exposition, for catching
a flaw in the original proof of Theorem \ref{thm:LABunion}, and for
observations that led to Theorem \ref{thm:supportsandclosedsets}.

\end{document}